\documentclass[11pt]{article}
\usepackage[tbtags]{amsmath}
\usepackage{amssymb}
\usepackage{amsthm}
\usepackage[misc]{ifsym}
\usepackage{cases}
\usepackage{mathrsfs}
\usepackage{graphicx}
\usepackage{subcaption} 

\numberwithin{equation}{section}
\normalsize
\title{\bf Decentralized Strategies for Backward Linear-Quadratic Mean Field Games and Teams \thanks{This work is supported by National Key R\&D Program of China (2022YFA1006104), National Natural Science Foundations of China (12471419, 12271304), and Shandong Provincial Natural Science Foundations (ZR2024ZD35, ZR2022JQ01).}}
\author{\normalsize  Yu Si\thanks{\it School of Mathematics, Shandong University, Jinan 250100, P.R. China, E-mail: 202112003@mail.sdu.edu.cn} , Jingtao Shi\thanks{\it Corresponding author. School of Mathematics, Shandong University, Jinan 250100, P.R. China, E-mail: shijingtao@sdu.edu.cn}}
\date{}
\newtheorem{Proposition}{Proposition}[section]
\newtheorem{Theorem}{Theorem}[section]

\newtheorem{Lemma}{Lemma}[section]
\newtheorem{Remark}{Remark}[section]
\newtheorem{Assumption}{Assumption}[section]
\newtheorem{Problem}{Problem}[section]
\begin{document}

\maketitle

\noindent{\bf Abstract:}\quad This paper studies a new class of linear-quadratic mean field games and teams problem, where the large-population system satisfies a class of $N$ weakly coupled linear backward stochastic differential equations (BSDEs), and $z_i$ (a part of solution of BSDE) enter the state equations and cost functionals. By virtue of stochastic maximum principle and optimal filter technique, we obtain a Hamiltonian system first, which is a fully coupled forward-backward stochastic differential equation (FBSDE). Decoupling the Hamiltonian system, we derive a feedback form optimal strategy by introducing Riccati equations, stochastic differential equation (SDE) and BSDE. Finally, we provide a numerical example to simulate our results.

\vspace{2mm}

\noindent{\bf Keywords:}\quad Mean field games, mean field team, backward stochastic differential equation, finite population, optimal filtering

\vspace{2mm}

\noindent{\bf Mathematics Subject Classification:}\quad 93E20, 60H10, 49K45, 49N70, 91A23

\section{Introduction}

Recently, the study of dynamic optimization in stochastic large-population systems has garnered significant attention. Distinguishing it from a standalone system, a large-population system comprises numerous agents, widely applied in fields such as engineering, finance and social science. In this context, the impact of a single agent is minimal and negligible, whereas the collective behaviors of the entire population are significant. All the agents are weakly coupled via the state average or empirical distribution in dynamics and cost functionals. Consequently, centralized strategies for a given agent, relying on information from all peers, are impractical. Instead, an effective strategy is to investigate the associated {\it mean field games} (MFGs) to identify an approximate equilibrium by analyzing its limiting behavior. Along this research direction, we can obtain the decentralized strategies through the limiting auxiliary control problems and the related {\it consistency condition} (CC) system. The past developments have largely followed two routes. One route starts by formally solving an $N$-agents game to obtain a large coupled solution equation system. The next step is to derive a limit for the solution by taking $N \rightarrow \infty$ \cite{Lasry-Lions-2007}, which is called the direct (or bottom-up) approach. The interested readers can refer to \cite{Cong-Shi-2024, Huang-Zhou-2020, Wang-2024}. Another route is to solve an optimal control problem of a single agent by replacing the state average term with a limiting process and formalize a fixed point problem to determine the limiting process, and this is called the fixed point (or top-down) approach. This kind of method is also called {\it Nash Certainty Equivalence} (NCE) \cite{Huang-Caines-Malhame-2007, Huang-Malhame-Caines-2006}. The interested readers can refer to \cite{Bardi-Priuli-2014, Bensoussan-Feng-Huang-2021, Huang-Huang-2017, Hu-Huang-Li-2018, Hu-Huang-Nie-2018, Huang-Wang-2016, Moon-Basar-2017, Nguyen-Huang-2012}. Further analysis of MFGs and related topics can be seen in \cite{Bensoussan-Frehse-Yam-2013, Buckdahn-Djehiche-Li-Peng-2009, Carmona-Delarue-2013, Huang-2010, Li-Sun-Xiong-2019} and the references therein. 

Different from noncooperative games, social optima, which is also called team decision problem, belong to the branch of the cooperative games. The participates seek optimal decision to minimize a common social cost functional together. Huang et al. \cite{Huang-Caines-Malhame-2012} firstly studied a class of linear-quadratic-Gaussian control problems with $N$ decision makers, where the basic objective is to minimize a social cost as the sum of $N$ individual costs containing mean field coupling. Feng et al. \cite{Feng-Huang-Wang-2021} studied a class of backward linear-quadratic Gaussian social optimization problems. Recent research on social team decision problems, see \cite{Arabneydi-Mahajan-2014, Feng-Lin-2024, Feng-Wang-2024, Wang-Zhang-2017, Wang-Zhang-Zhang-2020, Wang-Zhang-Zhang-2022}.

A BSDE is an SDE with a given random terminal value. As a consequence, the solution to BSDE should consist of one adapted pair $(y(\cdot),z(\cdot))$. Here, the second component $z(\cdot)$ is introduced to ensure the adaptiveness of $y(\cdot)$. The linear BSDE was firstly introduced by \cite{Bismut-1978}. Then, \cite{Pardoux-Peng-1990} generalized it to the nonlinear case. In fact, mean field problems driven by backward systems can be used to solve economic models with recursive utilities and cooperative relations (\cite{Du-Huang-Wu-2018, Feng-Huang-Wang-2021,  Huang-Wang-Wu-2016, Huang-Wang-Wang-Wang-2023, Li-Wu-2023}). One example is the production planning problem for a given minimum terminal, where the goal is to maximize the sum of product revenue. Another example is the hedging model of pension funds. In this case, we usually consider many types of pension funds and minimize the sum of the model risks. 

In this paper, we consider a class of linear-quadratic mean field games and team problem with backward stochastic large-population system. Inspired by research \cite{Wang-Zhang-Fu-Liang-2023} on finite populations, we utilize the backward separation method that arises in partial information problems to solve backward mean field problems. Compared with the existing literature, the contributions of this paper are listed as follows. 
\begin{itemize}
	\item The linear-quadratic mean field games and teams problems are introduced to a more general class of weakly-coupled backward stochastic system. In previous literature on MFGs with backward dynamics, only the second variable preceding the common noise could enter the generator of the state equation, such as \cite{Huang-Wang-Wang-Wang-2023}. However, in this paper, we have achieved the inclusion of the second variable $z_i$ preceding the private noise into the generator of the state equation. 
	\item Different from the fixed-point approach and direct approach commonly used to address large population problems, we adopts the backward separation approach (\cite{Wang-Wu-Xiong-2018}) to solve the problem. Firstly, we apply stochastic maximum principle and optimal filter technique to solve a multi-agent problem to get the decentralized strategy directly. And then, we introduce some Riccati equations, an SDE and a BSDE to obtain linear feedback form of decentralized strategies. 
	\item By utilizing the approach of backward separation, we can address not only traditional large-population problem but also problems involving finite or moderate-sized populations. Therefore, its application prospects are even broader.
	\item We give numerical simulations of the optimal state and optimal decentralized strategy to demonstrate the feasibility of our theoretical results.
\end{itemize}

The rest of this paper is organized as follows. In Section 2, we formulate our problem. In Section 3, we design the decentralized Nash equilibrium strategy of game problem. In Section 4, we design the decentralized social optimal strategy of social optima problem. In Section 5, we give a numerical example of theoretical results. Finally, the conclusion is given in Section 6.

\section{Problem formulation}

Firstly, we introduce some notations that will be used throughout the paper. We consider a finite time interval $[0, T]$ for a fixed $T > 0$.
Let $\big(\Omega, \mathcal{F}, \left\{\mathcal{F}_t\right\}_{t\geq0}, \mathbb{P}\big)$ be a complete filtered probability space, on which a standard $N$-dimensional Brownian motion $\big\{W_k(s), 1 \leqslant k \leqslant N\big\}_{0 \leqslant s \leqslant t}$ is defined, and  $\{\mathcal{F}_t\}$ is defined as the complete information of the system at time $t$. Then, for any $t \leqslant 0$, we have
$$
\mathcal{F}_t:=\sigma\left\{W_k(s), 1 \leqslant k \leqslant N, 0 \leqslant s \leqslant t\right\}\text{ and } \mathcal{F}_t^k:=\sigma\left(W_k(s), 0 \leqslant s \leqslant t \right),\ k=1,\cdots,N.
$$

Let $\mathbb{R}^n$ be an $n$-dimensional Euclidean space with norm and inner product being defined as $|\cdot|$ and $\langle\cdot, \cdot\rangle$, respectively.
Next, we introduce three spaces. A bounded, measurable function $f(\cdot):[0, T] \rightarrow \mathbb{R}^n$ is denoted as $f(\cdot) \in L^{\infty}(0, T; \mathbb{R}^n)$. An $\mathbb{R}^n$-valued, $\mathcal{F}_t$-adapted stochastic process $f(\cdot): \Omega \times [0, T] \rightarrow \mathbb{R}^n$ satisfying $\mathbb{E} \int_0^T |f(t)|^2 dt < \infty$ is denoted as $f(\cdot) \in L_{\mathcal{F}}^2(0, T; \mathbb{R}^n)$. Similarly, an $\mathbb{R}^n$-valued, $\mathcal{F}_{T}$-measurable random variable $\xi$ with $\mathbb{E} \xi^2 < \infty$ is denoted as $\xi \in L_{\mathcal{F}_T}^2(\Omega, \mathbb{R}^n)$.

For any random variable or stochastic process $X$ and filtration $\mathcal{H}$, $\mathbb{E}X$ represent the mathematical expectation of $X$. For a given vector or matrix \(M\), let \(M^{\top}\) represent its transpose. We denote the set of symmetric \(n \times n\) matrices (resp. positive semi-definite matrices) with real elements by \(\mathcal{S}^n\) (resp. \(\mathcal{S}_{+}^n\)). If \(M \in \mathcal{S}^n\) is positive (semi) definite, we abbreviate it as \(M > (\geqslant) 0\). For a positive constant \(k\), if \(M \in \mathcal{S}^n\) and \(M > kI\), we label it as \(M \gg 0\).

Now, let us focus on a large-population system comprised of $N$ individual agents, denoted as $\left\{\mathcal{A}_k\right\}_{1 \leqslant k \leqslant N}$. The state $x_k(\cdot)\in \mathbb{R}^n$ of agent $\mathcal{A}_k$ is given by the following linear $\mathrm{BSDE}$
\begin{equation}\label{state}
	\left\{\begin{aligned}
		d x_k(t)&=\Big[A(t) x_k(t)+B(t) u_k(t)+C(t) z_k(t)+f(t)\Big] d t+z_k(t) d W_k(t), \\
		x_k(T)&=\xi_k,
	\end{aligned}\right.
\end{equation}
where $u_k(\cdot) \in \mathbb{R}^r$ is the control process of agent $\mathcal{A}_k$, and $\xi_k \in L_{\mathcal{F}^i_T}^2(\Omega, \mathbb{R}^n)$ represents the terminal state, the coefficients $A(\cdot)$, $B(\cdot)$, $C(\cdot)$, $f(\cdot)$ are deterministic functions with compatible dimensions. Noting that $\{z_{k}(\cdot),1 \leqslant k \leqslant N\}$ are also the solution of (\ref{state}), which are introduced to ensure the adaptability of $x_k(\cdot)$.

For simplicity, let $u(\cdot)\equiv(u_1(\cdot),\cdots,u_N(\cdot))$ be the set of controls of all agents and let $u_{-k}(\cdot)\equiv(u_1(\cdot),\cdots,u_{k-1}(\cdot),u_{k+1}(\cdot),\cdots,u_N(\cdot))$ be the set of controls except for agent $\mathcal{A}_k$. The cost functional of agent $\mathcal{A}_k$ is given by
\begin{equation}\label{cost}
\begin{aligned}
	\mathcal{J}_k\left(u_k(\cdot);u_{-k}(\cdot)\right)&=\frac{1}{2} \mathbb{E}\left[\int_0^T\left\|x_k(t)-\Gamma_1(t) x^{(N)}(t)-\eta_1(t)\right\|_Q^2+\left\|u_k(t)\right\|_R^2+\left\|z_k \left(t\right)\right\|_H^2 d t\right] \\
	&\quad +\frac{1}{2} \mathbb{E}\left[\left\|x_k(0)-\Gamma_0 x^{(N)}(0)-\eta_0\right\|_G^2\right],
\end{aligned}
\end{equation}
where $Q(\cdot),R(\cdot),H(\cdot),\Gamma_1(\cdot)$, $\eta_1(\cdot)$ are deterministic functions and $\Gamma_0,\eta_0,G$ are vectors, with compatible dimensions. Define the admissible control set of agent $\mathcal{A}_k$ as 
$$
\mathcal{U}^{ad}_k:= \bigg\{u_k(\cdot) \mid u_k(t) \text { is adapted to } \mathcal{F}^k_t \text { and }
\mathbb{E} \int_0^T\left|u_k(t)\right|^2 d t<\infty\bigg\},\ k=1,\cdots,N.
$$

In this paper, we mainly study the following two problems:
\begin{Problem}(\textbf{PG})\label{game problem}
	Seek a Nash equilibrium strategy $u^*(\cdot)\equiv(u^*_1(\cdot),\ldots,u^*_N(\cdot))$, $u^*_k(\cdot) \in \mathcal{U}^{ad}_k$ for the system (\ref{state})-(\ref{cost}), i.e.,  $\mathcal{J}_k(u^*_k(\cdot);u^*_{-k}(\cdot)) = \inf_{u_k(\cdot) \in\ \mathcal{U}^{ad}_k} \mathcal{J}_k(u_k(\cdot);u^*_{-k}(\cdot))$, $k=1,\cdots,N$.
\end{Problem}

\begin{Problem}(\textbf{PS})\label{social problem}
	Seek a social optimal strategy $u^*(\cdot)\equiv(u^*_1(\cdot),\ldots,u^*_N(\cdot))$, $u^*_k(\cdot) \in \mathcal{U}^{ad}_k$ for the system (\ref{state})-(\ref{cost}), i.e., $\mathcal{J}_{soc}(u^*(\cdot)) =  \inf_{u_k(\cdot) \in\ \mathcal{U}^{ad}_k,\ k=1,\cdots,N} \mathcal{J}_{soc}(u(\cdot))$, where $\mathcal{J}_{soc}(u(\cdot)):=\sum_{k=1}^N\mathcal{J}_{k}(u(\cdot))\equiv\sum_{k=1}^N\mathcal{J}_{k}(u_k(\cdot);u_{-k}(\cdot))$.
\end{Problem}

We make the assumption on terminal states of agents.
\begin{Assumption}\label{A1}
	The terminal conditions $\left\{\xi_k \in L_{\mathcal{F}^k_T}^2\left(\Omega ; \mathbb{R}^n\right), k=1,2, \cdots, N\right\}$ are identically distributed and mutually independent. There exists a constant $c_0$ (independent of $N$) such that $\max _{1 \leqslant k \leqslant N} \mathbb{E}\left|\xi_k\right|^2<c_0$.
\end{Assumption}
In the following context, the time variable $t$ may be suppressed when no confusion occurs.

\section{Mean field game problem}

In this section, we study the linear-quadratic mean field game problem (\textbf{PG}). First, we introduce the following assumption:
\begin{Assumption}\label{A2 of game}
	The coefficients satisfy the following conditions:
	
	(i) $A(\cdot),C(\cdot),\Gamma_1(\cdot) \in L^\infty\left(0, T ; \mathbb{R}^{n \times n}\right)$, and $B(\cdot) \in L^\infty(0, T$; $\left.\mathbb{R}^{n \times r}\right) ;$
	
	(ii) $Q(\cdot),H(\cdot) \in L^\infty\left(0, T ; \mathbb{S}^n\right)$, $R(\cdot) \in L^{\infty}\left(0, T ; \mathbb{S}^r\right)$, and $R(\cdot)>0$, $Q(\cdot) \geqslant 0$, $H(\cdot) \geqslant 0 ;$
	
	(iii) $f(\cdot),\eta_1(\cdot) \in L^2\left(0, T ; \mathbb{R}^n\right) ;$
	
	(iv) $\Gamma_0 \in \mathbb{R}^{n \times n}$, $\eta_0 \in \mathbb{R}^{n}$, $G \in \mathbb{S}^n$ are bounded and $G \geqslant 0$.
\end{Assumption}

We next obtain the necessary and sufficient conditions for the existence of Nash equilibrium strategy of Problem (\ref{game problem}) by using variational analysis.
\begin{Theorem}
	Assume (\ref{A1}) and (\ref{A2 of game}) hold. Then Problem (\ref{game problem}) has a Nash equilibrium strategy $u^*(\cdot)\equiv(u^*_1(\cdot),\cdots,u^*_N(\cdot))$, $u^*_k(\cdot) \in \mathcal{U}^{ad}_k$ if and only if the following Hamiltonian system admits a set of solutions $(x^*_k(\cdot),z^*_k(\cdot),p^*_k(\cdot) :$ 
	\begin{equation}\label{Hamiltonian system of game}
		\left\{\begin{aligned}
			d x^*_k&=\left(A x^*_k+B u^*_k+C z^*_k+f\right) d t+z^*_k d W_k, \\
			d p^*_k&=-\left[A^\top p^*_k+\left(I_N-\frac{\Gamma_1}{N}\right)^\top Q\left(x^*_k-\Gamma_1 x^{*(N)}-\eta_1\right)\right] d t-\left(C^\top p^*_k+H z^*_k\right) d W_k, \\
			x^*_k(T)&=\xi_k,\quad p^*_k(0)=-\left(I_N-\frac{\Gamma_0}{N}\right)^\top G\left(x^*_k(0)-\Gamma_0 x^{*(N)}(0)-\eta_0\right),
		\end{aligned}\right.
	\end{equation}
	and the optimal strategy $u^*_k(\cdot)$ of agent $\mathcal{A}_k$ satisfies the
	stationary condition
	\begin{equation}\label{optimality conditions of game}
		u^*_k=-R^{-1}B^\top  \mathbb{E}\left[p^*_k \mid \mathcal{F}_t^k\right],\ k=1,\cdots,N.
	\end{equation}
\end{Theorem}

\begin{proof}
	Suppose $u^*(\cdot)$ is a Nash equilibrium strategy of Problem (\ref{game problem}) and $(x^*_k(\cdot),z^*_{k}(\cdot),k=1,\cdots,N)$ are the corresponding optimal trajectories. For any $u_k(\cdot) \in \mathcal{U}^{ad}_k$ and $\varepsilon>0$, we denote 
	$$
	u_k^\varepsilon(\cdot)=u^*_k(\cdot)+\varepsilon v_k(\cdot) \in \mathcal{U}^{ad}_k,
	$$
	where $v_k(\cdot)=u_k(\cdot)-u^*_k(\cdot)$.
	
	Let $(x^\epsilon_k(\cdot),z^\epsilon_k(\cdot),k=1,\cdots,N)$ be the solution of the following perturbed state equation:
	$$
	\left\{\begin{aligned}
		dx_k^{\varepsilon}=&\left(A x_k^{\varepsilon}+B u_k^{\varepsilon}+C z_k^{\varepsilon}+f\right) d t+z_k^{\varepsilon} d W_k, \\
		x_k^{\varepsilon}(T)&=\xi_k.
	\end{aligned}\right.
	$$
	Let $\Delta x_k(\cdot)=\frac{x_k^\varepsilon(\cdot)-\bar{x}_k(\cdot)}{\varepsilon},  \Delta z_k(\cdot)=\frac{z_{k}^\varepsilon(\cdot)-\bar{z}_k(\cdot)}{\varepsilon}$. It can be verified that $(\Delta x_k(\cdot),\Delta z_k(\cdot),k=1,\cdots,N)$ satisfies
	$$
	\left\{\begin{aligned}
		d \Delta x_k=&\left(A \Delta x_k+B v_k+C \Delta z_k\right) d t+\Delta z_k d W_k, \\
		\Delta x_k(T)&=0.
	\end{aligned}\right.
	$$
	Applying It\^o’s formula to $\left\langle \Delta x_k(\cdot), p^*_k(\cdot)\right\rangle$, we derive
	\begin{equation}\label{Ito maximum of game}
		\begin{aligned}
			&\mathbb{E}\left[0-\left\langle \Delta x_k(0),\left(I_n-\frac{\Gamma_0}{N}\right)^{\top} G\left(x^*_k(0)-\Gamma_0 x^{*(N)}(0)-\eta_0\right)\right\rangle\right]\\
			&=\mathbb{E} \int_0^T\left[\left\langle B v_k, p^*_k\right\rangle-\left\langle \Delta x_k,\left(I_n-\frac{\Gamma_1}{N}\right)^T Q\left(x^*_k-\Gamma_1 x^{*(N)}-\eta_1\right)\right\rangle-\left\langle \Delta z_{k}, H z^*_{k}\right\rangle\right] d t.
		\end{aligned}
	\end{equation}
	Then
	$$
	\begin{aligned}
		& \mathcal{J}_k\left(u_k^\varepsilon(\cdot) ; u^*_{-k}(\cdot)\right)-\mathcal{J}_k\left(u^*_k(\cdot) ; u^*_{-k}(\cdot)\right)=\frac{\varepsilon^2}{2} X_1+\varepsilon X_2,
	\end{aligned}
	$$
	where 
	\begin{equation*}
	\begin{aligned}
		X_1&= \mathbb{E}\left[\int_0^T\left[ \left((I_n-\frac{\Gamma_1}{N})  \Delta x_k\right)^\top Q\left((I_n-\frac{\Gamma_1}{N}) \Delta x_k\right)+ v_k^\top R v_k+ \Delta z_{k}^\top H \Delta z_{k}\right] d t\right. \\
		&\quad \left.\quad+\left((I_n-\frac{\Gamma_0}{N}) \Delta x_k(0)\right)^\top G\left((I_n-\frac{\Gamma_0}{N}) \Delta x_k(0)\right)\right],
	\end{aligned}
	\end{equation*}
	\begin{equation}\label{X_2 of game}
	\begin{aligned}
			X_2&= \mathbb{E} \left[\int_0^T\left[\left(x^*_k-\Gamma_1 x^{*(N)}-\eta_1\right)^\top Q\left((I_n-\frac{\Gamma_1}{N})  \Delta x_k\right)+u_k^{*\top} R v_k+z_{k}^{*\top} H  \Delta z_{k}\right] d t\right. \\
			&\quad \left.\quad+\left(x_k^*(0)-\Gamma_0 x^{*(N)}(0)-\eta_0\right)^\top G\left((I_n-\frac{\Gamma_0}{N})  \Delta x_k(0)\right)\right].
	\end{aligned}
	\end{equation}
	Due to the optimality of $u^*_k(\cdot)$, we have $\mathcal{J}_k\left(u_k^\varepsilon(\cdot) ; u^*_{-k}(\cdot)\right)-\mathcal{J}_k\left(u^*_k(\cdot) ; u^*_{-k}(\cdot)\right) \geqslant 0$. Noticing $X_1 \geqslant 0$ and the arbitrariness of $\epsilon$, we have $X_2=0$. Then, simplifying (\ref{X_2 of game}) with (\ref{Ito maximum of game}), we have
	$$
	X_2=\mathbb{E} \int_0^T\left\langle B^\top p^*_k+R u^*_k, v_k\right\rangle d t.
	$$
	Due to the arbitrariness of $v_k(\cdot)$, we obtain the optimal conditions (\ref{optimality conditions of game})
\end{proof}

Note that optimality conditions (\ref{optimality conditions of game}) are an open-loop Nash equilibrium strategy. The next step is to obtain proper form for the feedback representation of optimality conditions. We first introduce a lemma.

\begin{Lemma}\label{Lemma1}
	For any $j \neq k$, the following holds:
	\begin{equation}\label{lemma1 of game}
        \mathbb{E}\left[x^*_j(t) \mid \mathcal{F}_t^k\right]=\mathbb{E}\left[x^*_j(t)\right]=\mathbb{E}\left[x^*_k(t)\right] .
    \end{equation}
\end{Lemma}
\begin{proof}
	Note that $x^*_k(t)$ is adapted to $\mathcal{F}_t^k, k=1, \cdots, N$. Since $W_k(\cdot), k=$ $1, \cdots, N$, are mutually independent, by Assumption (\ref{A1}) then $x^*_k(\cdot)$, $k=1, \cdots, N$, are independent of each other. Since all agents have the same parameters, we obtain (\ref{lemma1 of game}).
\end{proof}
By Lemma \ref{Lemma1}, we have
$$
\mathbb{E}\left[x^{*(N)} \mid \mathcal{F}_t^k\right]=\frac{1}{N} x^*_k+\frac{1}{N} \sum_{j \neq k } \mathbb{E}x^*_j=\frac{1}{N} x^*_k+\frac{N-1}{N}  \mathbb{E}x^*_k.
$$
And then, we introduce the following FBSDE:
\begin{equation}\label{Hamiltonian system of game change}
	\left\{\begin{aligned}
		d x^*_k&=\left(A x^*_k-BR^{-1}B^\top \hat{p}^*_k+C z^*_k+f\right) d t+z^*_k d W_k, \\
		d \hat{p}^*_k&=-\left\{A^\top \hat{p}^*_k+\left(I_N-\frac{\Gamma_1}{N}\right)^\top Q\left[\left(I_N-\frac{\Gamma_1}{N}\right)x^*_k-\frac{N-1}{N}\Gamma_1 \mathbb{E} x^*_k-\eta_1\right]\right\} d t\\
		&\quad -\left(C^\top \hat{p}^*_k+H z^*_k\right) d W_k, \\
		x^*_k(T)&=\xi_k,\quad \hat{p}^*_k(0)=-\left(I_N-\frac{\Gamma_0}{N}\right)^\top G\left[\left(I_N-\frac{\Gamma_0}{N}\right)x^*_k(0)-\frac{N-1}{N}\Gamma_0 \mathbb{E} x^*_k(0)-\eta_0\right],
	\end{aligned}\right.
\end{equation}
where $\hat{p}^*_k(\cdot)=\mathbb{E}\left[p^*_k(\cdot) \mid \mathcal{F}_t^k\right]$. Then, we have the following proposition.

\begin{Proposition}\label{FBSDE equivalent theorem}
	The FBSDE (\ref{Hamiltonian system of game}) admits a set of adapted solutions if and only if FBSDE (\ref{Hamiltonian system of game change}) admits a set of adapted solutions.
\end{Proposition}
\begin{proof}
	If (\ref{Hamiltonian system of game change}) admits an adapted solution, then $(u^*_k(\cdot), k = 1,\ldots, N)$ is given and (\ref{Hamiltonian system of game}) is decoupled. Then (\ref{Hamiltonian system of game}) admits a solution. Conversely, if (\ref{Hamiltonian system of game change}) has a solution, then by taking the conditional expectation of the second equation in equation (\ref{Hamiltonian system of game}) and noting lemma 3.1, we obtain the solution of (\ref{Hamiltonian system of game change}). Here, we have applied lemma 5.4 of \cite{Xiong-2008}:
	$$
	\mathbb{E}\left[\int_0^t p_k^* d W_k \mid \mathcal{F}_t^k\right]=\int_0^t \mathbb{E}\left[p_k^* \mid \mathcal{F}_t^k\right] d W_k.
	$$

\end{proof}

Since the state equation is backward, we divide the decoupling procedure into two steps, inspired by \cite{Lim-Zhou-2001}.
\begin{Proposition}
	Let Assumption (\ref{A1}), (\ref{A2 of game}) hold. Let $(x^*_k(\cdot),z^*_k(\cdot),p^*_k(\cdot),k=1,\cdots,N)$ be the solution of FBSDE (\ref{Hamiltonian system of game change}). Then, we have the following relations:
	\begin{equation}
		\left\{\begin{aligned}
			x^*_k&=\Sigma \left(\hat{p}^*_k-\mathbb{E} p^*_k\right)+K \mathbb{E} p^*_k+\varphi^*_k,\\
			z^*_k&=\left(I_n+\Sigma H\right)^{-1} \left(\beta^*_k-\Sigma C^\top \hat{p}^*_k\right),
		\end{aligned}\right.
	\end{equation}
	and $\Sigma(\cdot),K(\cdot),\varphi^*_k(\cdot)$ are solutions of the following equations, respectively: 
	\begin{equation}\label{Sigma equation}
		\left\{\begin{aligned}
			&\dot{\Sigma}-A \Sigma-\Sigma A^\top-\Sigma\left(I_n-\frac{\Gamma_1}{N}\right)^\top Q \left(I_n-\frac{\Gamma_1}{N}\right) \Sigma+B R^{-1} B^\top\\
			&+C \left(I_n+\Sigma H\right)^{-1} \Sigma C^\top=0, \\
			&\Sigma(T)=0,
		\end{aligned}\right.
	\end{equation}
	\begin{equation}\label{K equation}
		\left\{\begin{aligned}
			& \dot{K}-A K-K A^\top-K\left(I_n-\frac{\Gamma_1}{N}\right) ^\top Q (I_n-\Gamma_1) K+B R^{-1} B^\top\\
			&+C \left(I_n+\Sigma H\right)^{-1} \Sigma C^\top=0, \\
			& K(T)=0,
		\end{aligned}\right.
	\end{equation}
	\begin{equation}\label{varphi equation}
		\left\{\begin{aligned}
			d \varphi^*_k&=\left\{\left[A+\Sigma\left(I_n-\frac{\Gamma_1}{N}\right)^\top Q \left(I_n-\frac{\Gamma_1}{N}\right)\right] \varphi^*_k +\left[K\left(I_n-\frac{\Gamma_1}{N}\right)^\top Q (I_n-\Gamma_1) \right.\right.\\
			&\quad \left.-\Sigma\left(I_n-\frac{\Gamma_1}{N}\right)^\top Q \left(I_n-\frac{\Gamma_1}{N}\right)\right] \mathbb{E}\varphi^*_k+C\left(I_n+\Sigma H\right)^{-1}\beta^*_k\\
			&\quad \left.-K\left(I_n-\frac{\Gamma_1}{N}\right)^{\top} Q \eta_1+f\right\}dt+ \beta^*_k d W_k,\\
			\varphi^*_k(T)&=\xi_k.
		\end{aligned}\right.
	\end{equation}
\end{Proposition}

\begin{proof}
	Noting the terminal condition and structure of (\ref{Hamiltonian system of game change}), for each $k = 1,\cdots, N$, we suppose
	\begin{equation}\label{decouple form of game 1}
		x^*_k=\Sigma \left(p^*_k-\mathbb{E}p^*_k\right)+K \mathbb{E}p^*_k+\varphi^*_k,
	\end{equation}
	with $\Sigma(T) = 0, K(T) = 0$ for two $\mathbb{R}^{n \times n}$-valued deterministic differentiable functions $\Sigma(\cdot), K(\cdot)$, and for an $\mathbb{R}^n$-valued $\mathcal{F}_t^i$-adapted process $\varphi^*_k(\cdot)$ satisfying a BSDE:
	\begin{equation}
		\left\{\begin{aligned}
			d \varphi^*_k&=\alpha_k^* d t+  \beta^*_k d W_k,\\
			\varphi^*_k(T)&=\xi_k.
		\end{aligned}\right.
	\end{equation}
	Then, we have
	\begin{equation}\label{decouple form of game 2}
	x^*_k-\mathbb{E}x^*_k=\Sigma\left(\hat{p}^*_k-\mathbb{E}p^*_k\right)+\left(\varphi^*_k-\mathbb{E} \varphi^*_k\right), 
	\end{equation}
	\begin{equation}\label{decouple form of game 3}
	\mathbb{E}x^*_k=K \mathbb{E}p^*_k+\mathbb{E} \varphi^*_k.
    \end{equation}
	Applying It\^o’s formula to (\ref{decouple form of game 2}), we have
    $$
    \begin{aligned}
    	 d\left(x^*_k-\mathbb{E}x^*_k\right)&=\dot{\Sigma}\left(\hat{p}^*_k-\mathbb{E}p^*_k\right)dt-\Sigma\left(C^\top \hat{p}^*_k+H z^*_k\right)dW_k+\left(\alpha^*_k-\mathbb{\alpha}^*_k\right)+\beta_k dW_k\\
    	 &\quad -\Sigma \left[A^\top \left(\hat{p}^*_k-\mathbb{E}p^*_k\right)+\left(I_n-\frac{\Gamma_1}{N}\right)^\top Q \left(I_n-\frac{\Gamma_1}{N}\right)\left(x^*_k-\mathbb{E}x^*_k\right)\right] d t \\
    	 &=\left[A\left(x^*_k-\mathbb{E}x^*_k\right)-B R^{-1} B^{\top}\left(\hat{p}^*_k-\mathbb{E}p^*_k\right)+C\left(z^*_k-\mathbb{E}z^*_k\right)\right] d t+z^*_k d W_k.
    \end{aligned}
    $$
	By comparing the coefficients of the diffusion terms, we obtain
	$$
	\beta^*_k-\Sigma C^\top \hat{p}^*_k-\Sigma H z^*_k-z^*_k=0.
	$$
	Then we can solve for $z^*_k(\cdot)$ explicitly:
	\begin{equation}\label{z_k}
		z^*_{k}=\left(I_n+\Sigma H\right)^{-1} \left(\beta^*_{k}-\Sigma C^{\top} \hat{p}^*_{k}\right).
	\end{equation}
	Then by comparing the coefficients of the drift terms and noting (\ref{decouple form of game 2}) and (\ref{z_k}), we obtain 
	\begin{equation}\label{drift term of game}
	\begin{aligned}
		& \left[\dot{\Sigma}-\Sigma A^\top -A \Sigma-\Sigma\left(I_n-\frac{\Gamma_1}{N}\right)^\top Q\left(I_n-\frac{\Gamma_1}{N}\right) \Sigma+B R^{-1} B^\top\right.\\
        &\quad +C\left(I_n+\Sigma H\right)^{-1}\Sigma C^\to p\bigg]\left(\hat{p}^*_k-\mathbb{E}\hat{p}^*_k\right)+ \left(\alpha^*_k-\mathbb{E} \alpha^*_k\right)\\
		& -\left[\Sigma\left(I_n-\frac{\Gamma_1}{N}\right)^\top Q\left(I_n-\frac{\Gamma_1}{N}\right)+A\right]\left(\varphi^*_k-\mathbb{E} \varphi^*_k\right)-C\left(I_n+\Sigma H\right)^{-1}\left(\beta^*_k-\mathbb{E}\beta^*_k\right)=0.
	\end{aligned}
    \end{equation}
	Combining (\ref{drift term of game}), we can obtain the equation (\ref{Sigma equation}) of $\Sigma(\cdot)$  from the coefficients of $\hat{p}^*_k(\cdot)-\mathbb{E}\hat{p}^*_k(\cdot)$ and 
	\begin{equation}\label{drift non-homogeneous term of game 1}
		\begin{aligned}
		& \left(\alpha^*_k-\mathbb{E} \alpha^*_k\right)-\left[\Sigma\left(I_n-\frac{\Gamma_1}{N}\right)^\top Q\left(I_n-\frac{\Gamma_1}{N}\right)+A\right]\left(\varphi^*_k-\mathbb{E} \varphi^*_k\right)\\
        & -C\left(I_n+\Sigma H\right)^{-1}\left(\beta^*_k-\mathbb{E}\beta^*_k\right)=0.
		\end{aligned}
	\end{equation}
	Then, applying It\^o’s formula to (\ref{decouple form of game 3}), we have
	$$
	\begin{aligned}
		d \mathbb{E}x^*_k&= \dot{K} \mathbb{E}p^*_k d t-K\left\{A^\top \mathbb{E}p^*_k+\left(I_n-\frac{\Gamma_1}{N}\right) Q\left[\left(I_n-\Gamma_1\right) \mathbb{E}x^*_k-\eta_1\right]\right\} d t +\mathbb{E} \alpha^*_k d t \\
		&= \left(A \mathbb{E}x^*_k-B R^{-1} B^\top \mathbb{E}p^*_k+C \mathbb{E}z^*_k+f\right) d t.
	\end{aligned}
	$$
	Then noting (\ref{decouple form of game 3}) and (\ref{z_k}), we obtain 
	\begin{equation}\label{E term of game}
	\begin{aligned}
		& \left(\dot{K}-K A^{\top}-A K+B R^{-1} B^\top -K\left(I_n-\frac{\Gamma_1}{N}\right)^\top Q\left(I_n-\Gamma_1\right) K \right.\\
		&\quad +C\left(I_n+\Sigma H\right)^{-1} \Sigma C^\top \bigg) \mathbb{E} p^*_k -K\left(I_n-\frac{\Gamma_1}{N}\right)^\top Q\left(I_n-\Gamma_1\right) \mathbb{E} \varphi^*_k \\
		& +K\left(I_n-\frac{\Gamma_1}{N}\right)^\top Q \eta_1+\mathbb{E} \alpha_k^*-A \mathbb{E} \varphi^*_k -C\left(I_n+\Sigma H\right)^{-1} \mathbb{E} \beta^*_k-f=0.
	\end{aligned}
	\end{equation}
	Combining (\ref{E term of game}), we can obtain the equation (\ref{K equation}) of $K(\cdot)$  from the coefficients of $\mathbb{E}p^*_k(\cdot)$ and 
	\begin{equation}\label{drift non-homogeneous term of game 2}
	\begin{aligned}
		&-K\left(I_n-\frac{\Gamma_1}{N}\right)^\top Q\left(I_n-\Gamma_1\right) \mathbb{E} \varphi^*_k +K\left(I_n-\frac{\Gamma_1}{N}\right)^\top Q \eta_1+\mathbb{E} \alpha_k^*-A \mathbb{E} \varphi^*_k \\
		&-C\left(I_n+\Sigma H\right)^{-1} \mathbb{E} \beta^*_k-f=0.
	\end{aligned}
	\end{equation}
	Combining (\ref{drift non-homogeneous term of game 1}) and (\ref{drift non-homogeneous term of game 2}), we have 
	\begin{equation}\label{drift non-homogeneous term of game 3}
	\begin{aligned}
		\alpha^*_k&= \left[A+\Sigma\left(I_n-\frac{\Gamma_1}{N}\right)^\top Q \left(I_n-\frac{\Gamma_1}{N}\right)\right] \varphi^*_k \\
		&\quad +\left[K\left(I_n-\frac{\Gamma_1}{N}\right)^\top Q (I_n-\Gamma_1) -\Sigma\left(I_n-\frac{\Gamma_1}{N}\right)^\top Q \left(I_n-\frac{\Gamma_1}{N}\right)\right] \mathbb{E}\varphi^*_k\\
        &\quad +C\left(I_n+\Sigma H\right)^{-1}\beta_k-K\left(I_n-\frac{\Gamma_1}{N}\right)^\top Q \eta_1+f,
	\end{aligned}
	\end{equation}
	and the equation (\ref{varphi equation}) of $\varphi^*_i(\cdot)$ can be got by using (\ref{drift non-homogeneous term of game 3}). Then, we completed the proof.
\end{proof}

\begin{Remark}
	Under Assumptions (\ref{A1})–(\ref{A2 of game}), the Riccati equations (\ref{Sigma equation}) and (\ref{K equation}) admit unique solutions respectively \cite{Li-Sun-Xiong-2019}. Once $\Sigma(\cdot)$ and $K(\cdot)$ are known, the existence of solution to (\ref{varphi equation}) is clearly
	established.
\end{Remark}

\begin{Proposition}\label{decouple proposition 2}
	Let Assumption (\ref{A1}), (\ref{A2 of game}) hold. Let $(x^*_k(\cdot),z^*_k(\cdot),p^*_k(\cdot),k=1,\cdots,N)$ be the solution of FBSDE (\ref{Hamiltonian system of game change}). Then, we have the following relations:
	\begin{equation}\label{decouple form of game 4}
		\hat{p}^*_k=\Pi \left(x^*_k-\mathbb{E}x^*_k\right)+M \mathbb{E}x^*_k+\zeta^*_k,
	\end{equation}
	where $\Pi(\cdot),M(\cdot),\zeta^*_k(\cdot)$ are solutions of the following equations, respectively: 
	\begin{equation}\label{Pi equation}
		\left\{\begin{aligned}
			&\dot{\Pi}+\Pi A^\top +A \Pi-\Pi\left[B R^{-1} B^\top +C\left(I_n+\Sigma H\right)^{-1}\Sigma C^\top \right]\Pi\\
			&+\left(I_n-\frac{\Gamma_1}{N}\right)^\top Q\left(I_n-\frac{\Gamma_1}{N}\right)=0, \quad
			\Pi(0)=-\left(I_n-\frac{\Gamma_0}{N}\right)^\top G\left(I_n-\frac{\Gamma_0}{N}\right),
		\end{aligned}\right.
	\end{equation}
	\begin{equation}\label{M equation}
		\left\{\begin{aligned}
			&\dot{M}+M A^\top +A M-M\left[BR^{-1}B^\top +C\left(I_n+\Sigma H\right)^{-1}\Sigma C^\top \right]M \\
			& +\left(I_n-\frac{\Gamma_1}{N}\right)^{\top} Q\left(I_n-\Gamma_1\right)=0,\quad
			M(0)=\left(I_n-\frac{\Gamma_0}{N}\right)^{\top} G \left(I_n-\Gamma_0\right),
		\end{aligned}\right.
	\end{equation}
	\begin{equation}\label{zeta equation}
		\left\{\begin{aligned}
			d \zeta^*_k&=\bigg\{\left\{\Pi\left[BR^{-1}B^\top +C\left(I_n+\Sigma H\right)^{-1}\Sigma C^\top \right]-A^\top \right\} \zeta^*_k-\Pi C\left(I_n+\Sigma H\right)^{-1}\beta^*_k \\
			&\qquad +\left(\Pi-M\right)C\left(I_n+\Sigma H\right)^{-1}\mathbb{E}\beta^*_k-Mf-\left(I_n-\frac{\Gamma_1}{N}\right)^\top Q \eta_1\\
			&\qquad +\left(M-\Pi\right)\left[BR^{-1}B^\top +C\left(I_n+\Sigma H\right)^{-1}\Sigma C^\top \right]\mathbb{E}\zeta^*_k \bigg\}dt\\
			&\quad +\bigg\{-\left(\Pi+H\right)\left(I_n+\Sigma H\right)^{-1}\beta^*_k-\left(I_n-\Pi\Sigma\right)\left(I_n+H\Sigma \right)^{-1}C^\top\\
			&\qquad\ \times\left\{\left(I_n-\Pi\Sigma\right)^{-1}\big[\Sigma\left(\varphi^*_k-\mathbb{E}\varphi^*_k\right)+\left(\zeta^*_k-\mathbb{E}\zeta^*_k\right)\big]\right.\\
            &\qquad\ \left.+\left(I_n-\Pi\Sigma\right)^{-1}\left(M\mathbb{E}\varphi^*_k +\mathbb{E}\zeta^*_k\right)\right\}\bigg\}dW_k,\quad \zeta^*_k(0)=\left(I_n-\frac{\Gamma_0}{N}\right)^\top G \eta_0.
		\end{aligned}\right.
	\end{equation}
\end{Proposition}
\begin{proof}
	By combining (\ref{optimality conditions of game}) with (\ref{decouple form of game 1}), BSDE in (\ref{Hamiltonian system of game change}) is rewritten as
	\begin{equation}
		\left\{\begin{aligned}
			d x^*_k&=\left[A x^*_k-\left(BR^{-1}B^\top +C\left(I_n+\Sigma H\right)^{-1}\Sigma C^\top \right) \hat{p}^*_k+(C\left(I_n+\Sigma H\right)^{-1} \beta^*_k+f\right] d t\\
			&\quad +C\left(I_n+\Sigma H\right)^{-1} \left(\beta^*_k-\Sigma C^\top \hat{p}^*\right) d W_k, \\
			x^*_k(0)&=\left[I_n-\Sigma\left(I_n-\frac{\Gamma_0}{N}\right)^\top G\left(I_n-\frac{\Gamma_0}{N}\right)\right]^{-1}\left(\varphi^*_k(0)-\mathbb{E}\varphi^*_k(0)\right)\\
			&\quad +\left[I_n-K\left(I_n-\frac{\Gamma_0}{N}\right)^\top G\left(I_n-\Gamma_0\right)\right]^{-1}\mathbb{E}\varphi^*_k(0)\\
			&\quad -K\left[I_n-\left(I_n-\frac{\Gamma_0}{N}\right)^{\top} G\left(I_n-\Gamma_0\right)K\right]^{-1}G\eta_0.
		\end{aligned}\right.
	\end{equation}
	Noting the initial condition and structure of (\ref{Hamiltonian system of game change}), for each $k = 1,\cdots, N$, we suppose
	\begin{equation}
		\hat{p}^*_k=\Pi \left(x^*_k-\mathbb{E}x^*_k\right)+M \mathbb{E}x^*_k+\zeta^*_k,
	\end{equation}
	with $\Pi(0) = -\left(I_n-\frac{\Gamma_0}{N}\right)^{\top} G\left(I_n-\frac{\Gamma_0}{N}\right), M(0) = \left(I_n-\frac{\Gamma_0}{N}\right)^{\top} G \left(I_n-\Gamma_0\right)$ for two deterministic differentiable functions $\Pi(\cdot), M(\cdot)$, and for an $\mathcal{F}_t^k$-adapted process $\zeta^*_k(\cdot)$ satisfying an SDE
	$$
	\left\{\begin{aligned}
		d \zeta^*_k&=\chi_k^* d t+  \gamma^*_k d W_k, \\
		\zeta^*_k(0)&=\left(I_n-\frac{\Gamma_0}{N}\right)^\top G \eta_0 .
	\end{aligned}\right.
	$$
	Then, we have
	\begin{equation}\label{decouple form of game 5}
		\hat{p}^*_k-\mathbb{E}p^*_k=\Pi\left(x^*_k-\mathbb{E}x^*_k\right)+\left(\zeta^*_k-\mathbb{E} \zeta^*_k\right), 
	\end{equation}
	\begin{equation}\label{decouple form of game 6}
		\mathbb{E}p^*_k=M \mathbb{E}x^*_{k}+\mathbb{E} \zeta^*_k.
	\end{equation}
	Applying It\^o’s formula to (\ref{decouple form of game 5}), we have
	$$
	\begin{aligned}
		d\left(\hat{p}^*_k-\mathbb{E}p^*_k\right)&=\dot{\Pi}\left(x^*_k-\mathbb{E}x^*_k\right)dt+\Pi\left[\left(I_n+\Sigma H\right)^{-1}\beta^*_k-\left(I_n+\Sigma H\right)^{-1}\Sigma C^\top \hat{p}^*_k\right]dW_k\\
		&\quad -\Pi \left[A \left(x^*_k-\mathbb{E}x^*_k\right)-\left(BR^{-1}B^\top +C\left(I_n+\Sigma H\right)^{-1}\Sigma C^\top \right)\left(\hat{p}^*_k-\mathbb{E}p^*_k\right)\right.\\
		&\quad \left.+C\left(I_n+\Sigma H\right)^{-1}\left(\beta^*_k-\mathbb{E}\beta^*_k\right)\right] d t +\left(\chi^*_k-\mathbb{\chi}^*_k\right)dt+\gamma^*_k dW_k\\
		&=-\left[A^\top \left[\Pi \left(x^*_k-\mathbb{E}x^*_k\right)+\left(\zeta_k-\mathbb{E}\zeta_k\right)\right]+\left(I_n-\frac{\Gamma_1}{N}\right)^\top Q\left(I_n-\frac{\Gamma_1}{N}\right)\left(x^*_k-\mathbb{E}x^*_k\right)\right]dt\\
		&\quad -\left[C^\top \hat{p}^*_k+H\left(I_n+\Sigma H\right)^{-1}\left(\beta^*_k-\Sigma C^\top \hat{p}^*_k\right)\right] d W_k.
	\end{aligned}
	$$
	By comparing the coefficients of the diffusion terms, we obtain
	$$
	\gamma^*_k+\Pi\left[\left(I_n+\Sigma H\right)^{-1}\beta^*_k-\left(I_n+\Sigma H\right)^{-1}\Sigma C^\top \hat{p}^*_k\right]+C^{\top}\hat{p}^*_k+H\left(I_n+\Sigma H\right)^{-1}\left(\beta^*_k-\Sigma C^\top \hat{p}^*_k\right)=0,
	$$
	Using the fact that
	$$
	\left(I_n+\Sigma H\right)^{-1}\Sigma=\Sigma\left(I_n+H\Sigma \right)^{-1},
	$$
	it can be shown by a straightforward computation that 
	$$
	\gamma^*_k+\left(\Pi+H\right)\left(I_n+\Sigma H\right)^{-1}\beta^*_k+\left(I_n-\Pi\Sigma\right)\left(I_n+H\Sigma \right)^{-1}C^\top \hat{p}^*_k=0.
	$$
	Combining equations (\ref{decouple form of game 1}) with (\ref{decouple form of game 4}), we can explicitly solve for $\hat{p}^*_k(\cdot)$:
	\begin{equation}\label{p_k equation}
	    \hat{p}^*_k=\left(I_n-\Pi\Sigma\right)^{-1}\left[\Sigma\left(\varphi^*_k-\mathbb{E}\varphi^*_k\right)+\left(\zeta^*_k-\mathbb{E}\zeta^*_k\right)\right]+\left(I_n-\Pi\Sigma\right)^{-1}\left(M\mathbb{E}\varphi^*_k+\mathbb{E}\zeta^*_k\right).
    \end{equation}
	By (\ref{p_k equation}), we can explicitly solve for $\gamma^*_k(\cdot)$:
	\begin{equation}\label{gamma_k equation}
	    \begin{aligned}
	        \gamma^*_k&=-\left(\Pi+H\right)\left(I_n+\Sigma H\right)^{-1}\beta^*_k-\left[\left(I_n+\Sigma H\right)^{-1}C^\top -\Pi\left(I_n+\Sigma H\right)^{-1}\Sigma C^\top \right]\\
	        &\quad \times\left\{\left(I_n-\Pi\Sigma\right)^{-1}\left[\Sigma\left(\varphi^*_k-\mathbb{E}\varphi^*_k\right)+\left(\zeta^*_k-\mathbb{E}\zeta^*_k\right)\right]+\left(I_n-\Pi\Sigma\right)^{-1}\left(M\mathbb{E}\varphi^*_k+\mathbb{E}\zeta^*_k\right)\right\}
	    \end{aligned}
    \end{equation}
	Then by comparing the coefficients of the drift terms and noting (\ref{decouple form of game 5}), we obtain
	\begin{equation}\label{drift term of game 2}
		\begin{aligned}
			& \bigg\{\dot{\Pi}+\Pi A^\top +A \Pi-\Pi\left[B R^{-1} B^\top +C\left(I_n+\Sigma H\right)^{-1}\Sigma C^\top \right]\Pi\\
			& +\left(I_n-\frac{\Gamma_1}{N}\right)^\top Q\left(I_n-\frac{\Gamma_1}{N}\right)\bigg\}\left(x^*_k-\mathbb{E}x^*_k\right)\\
            &-\Pi\left[BR^{-1}B^{\top}+C\left(I_n+\Sigma H\right)^{-1}\Sigma C^\top \right]\left(\zeta^*_k-\mathbb{E}\zeta^*_k\right) \\
			&+\Pi C\left(I_n+\Sigma H\right)^{-1}\left(\beta^*_k-\mathbb{E}\beta^*_k\right)
            +\left(\chi^*_k-\mathbb{E} \chi^*_k\right)-A\left(\zeta^*_k-\mathbb{E} \zeta^*_k\right)=0.
		\end{aligned}
	\end{equation}
	Combining (\ref{drift term of game 2}), we can obtain the equation (\ref{Pi equation}) of $\Pi(\cdot)$  from the coefficients of $x^*_k(\cdot)-\mathbb{E}x^*_k(\cdot)$ and 
	\begin{equation}\label{drift non-homogeneous term of game 4}
		\begin{aligned}
			\chi^*_k-\mathbb{E} \chi^*_k&=\left\{\Pi\left[BR^{-1}B^\top +C\left(I_n+\Sigma H\right)^{-1}\Sigma C^\top \right]-A^\top \right\}\left(\zeta^*_k-\mathbb{E}\zeta^*_k\right)\\
			&\quad -\Pi C\left(I_n+\Sigma H\right)^{-1}\left(\beta^*_k-\mathbb{E}\beta^*_k\right).
		\end{aligned}
	\end{equation}
	Then, applying It\^o’s formula to (\ref{decouple form of game 6}), we have
	$$
	\begin{aligned}
		d \mathbb{E}x^*_k&= \dot{M} \mathbb{E}x^*_k d t+M\left\{A \mathbb{E}x^*_k-\left[BR^{-1}B^\top +C\left(I_n+\Sigma H\right)^{-1}\Sigma C^\top \right]\left(M\mathbb{E}x^*_k+\mathbb{E}\zeta^*_k\right)\right.\\
		&\quad \left.+C\left(I_n+\Sigma H\right)^{-1}\mathbb{E}\beta^*_k+f\right\} d t +\mathbb{E} \chi^*_k d t \\
		&=-\left[A^\top \left(M\mathbb{E}x^*_k+\mathbb{E}\zeta^*_k\right)+\left(I_n-\frac{\Gamma_1}{N}\right)^\top Q\left(I_n-\Gamma_1\right)\mathbb{E}x^*_k-\left(I_n-\frac{\Gamma_1}{N}\right)^\top Q\eta_1 \right]d t.
	\end{aligned}
	$$
	Then by comparing the coefficients of the drift terms and noting (\ref{decouple form of game 5}), we obtain
	\begin{equation}\label{E term of game 2}
		\begin{aligned}
			& \left\{\dot{M}+M A^\top +A M+B R^{-1} B^\top -M\left[BR^{-1}B^\top +C\left(I_n+\Sigma H\right)^{-1}\Sigma C^\top \right]M \right.\\
			&\left. +\left(I_n-\frac{\Gamma_1}{N}\right)^\top Q\left(I_n-\Gamma_1\right)\right\} \mathbb{E} x^*_k -M\left[BR^{-1}B^\top +C\left(I_n+\Sigma H\right)^{-1}\Sigma C^\top \right] \mathbb{E}\zeta^*_k \\
			&+MC\left(I_n+\Sigma H\right)^{-1}\mathbb{E}\beta^*_k+Mf+\mathbb{E}\chi^*_k+A^\top \mathbb{E}\zeta^*_k+\left(I_n-\frac{\Gamma_1}{N}\right)^\top Q \eta_1=0.
		\end{aligned}
	\end{equation}
	Combining (\ref{E term of game 2}), we can obtain the equation (\ref{M equation}) of $M(\cdot)$  from the coefficients of $\mathbb{E}x^*_k(\cdot)$ and 
	\begin{equation}\label{drift non-homogeneous term of game 5}
		\begin{aligned}
			\mathbb{E}\chi^*_k&=\left\{M\left[BR^{-1}B^\top +C\left(I_n+\Sigma H\right)^{-1}\Sigma C^\top \right]-A^{\top}\right\} \mathbb{E}\zeta^*_k \\
			&\quad -MC\left(I_n+\Sigma H\right)^{-1}\mathbb{E}\beta^*_k-Mf-\left(I_n-\frac{\Gamma_1}{N}\right)^{\top} Q \eta_1=0.
		\end{aligned}
	\end{equation}
	Combining (\ref{drift non-homogeneous term of game 4}) and (\ref{drift non-homogeneous term of game 5}), we have 
	\begin{equation}\label{drift non-homogeneous term of game 6}
		\begin{aligned}
			\chi^*_k&=\left\{\Pi\left[BR^{-1}B^\top +C\left(I_n+\Sigma H\right)^{-1}\Sigma C^\top \right]-A^\top \right\} \zeta^*_k-\Pi C\left(I_n+\Sigma H\right)^{-1}\beta^*_k \\
			&\quad +\left(\Pi-M\right)C\left(I_n+\Sigma H\right)^{-1}\mathbb{E}\beta^*_k-Mf-\left(I_n-\frac{\Gamma_1}{N}\right)^\top Q \eta_1\\
			&\quad +\left\{\left(M-\Pi\left[BR^{-1}B^\top +C\left(I_n+\Sigma H\right)^{-1}\Sigma C^\top \right]\right)\right\}\mathbb{E}\zeta^*_k,
		\end{aligned}
	\end{equation}
	and the equation (\ref{zeta equation}) of $\zeta^*_k(\cdot)$ can be got by using (\ref{gamma_k equation}) and (\ref{drift non-homogeneous term of game 6}). Then, we completed the proof.
\end{proof}
\begin{Remark}
	Under Assumptions (\ref{A1})–(\ref{A2 of game}), the Riccati equations (\ref{Pi equation}) and (\ref{M equation}) admit unique solutions, respectively \cite{Lim-Zhou-2001}. Similarly, once $\Sigma(\cdot)$, $K(\cdot)$, $\Pi(\cdot)$, $M(\cdot)$, $(\varphi^*_k(\cdot),\beta^*_k(\cdot))$ are known, SDE (\ref{zeta equation}) has a unique solution $\zeta^*_k$.
\end{Remark}
The above discussion can be summarized as the following theorem.
\begin{Theorem}
	Let Assumptions (\ref{A1}), (\ref{A2 of game}) hold. Then Riccati equations (\ref{Sigma equation}), (\ref{K equation}), (\ref{Pi equation}), (\ref{M equation}), BSDE (\ref{varphi equation}) and SDE (\ref{zeta equation}) admit unique solutions, respectively. In addition, the optimal strategy of agent $\mathcal{A}_k$ has a feedback form as follows:
	\begin{equation}\label{optimal game strategy}
		u^*_k=-R^{-1} B^\top\big[\Pi x^*_k+\left(M-\Pi\right) \mathbb{E}x^*_k+\zeta^*_k\big],
	\end{equation}
	where
	\begin{equation}
		\left\{\begin{aligned}
			d x^*_k&=\left[\left(A-BR^{-1} B^\top \Pi\right) x^*_k-BR^{-1} B^\top \left(M-\Pi\right) \mathbb{E}x^*_k\right.\\
			&\qquad \left.-BR^{-1} B^\top \zeta^*_k+C z^*_k+f\right] d t+z^*_k d W_k, \\
			x^*_k(T)&=\xi_k,
		\end{aligned}\right.
	\end{equation}
	\begin{equation}
		\left\{\begin{aligned}
			d \mathbb{E}x^*_k&=\left\{\left(A-BR^{-1} B^\top M-C \left(I_n+\Sigma H\right)^{-1}\Sigma C^\top M\right) \mathbb{E}x^*_k \right.\\
			&\quad \left.-\left[BR^{-1} B^\top +C \left(I_n+\Sigma H\right)^{-1}\Sigma C^\top \right]\mathbb{E}\zeta^*_k+C \left(I_n+\Sigma H\right)^{-1} \mathbb{E}\beta^*_k+f\right\} d t \\
			\mathbb{E}x^*_k(T)&=\mathbb{E}\xi_k.
		\end{aligned}\right.
	\end{equation}
\end{Theorem}

\section{Mean field social optima problem}

\subsection{Optimal strategy of mean field social optima problem}

In this section, we study the linear-quadratic mean field social optima problem  (\textbf{PS}). For avoiding heavy notation, we denote
$$
\begin{aligned}
	&\left\{\begin{aligned}
		Q_{\Gamma}&=Q \Gamma_1+\Gamma_1^{\mathrm{T}} Q-\Gamma_1^{\mathrm{T}} Q \Gamma_1, \quad
		\bar{\eta}_1=Q \eta_1-\Gamma_1^{\mathrm{T}} Q \eta_1, \\
		G_{\Gamma}&=G \Gamma_0+\Gamma_0^{\mathrm{T}} G-\Gamma_0^{\mathrm{T}} G \Gamma_0, \quad
		\bar{\eta}_0=G \eta_0-\Gamma_0^{\mathrm{T}} G \eta_0.
	\end{aligned}\right.
\end{aligned}
$$
First, we introduce the following assumption:
\begin{Assumption}\label{A2 of social}
	(i) $A(\cdot),C(\cdot), \Gamma_1(\cdot) \in L^\infty\left(0, T ; \mathbb{R}^{n \times n}\right)$, and $B(\cdot) \in L^{\infty}(0, T$; $\left.\mathbb{R}^{n \times r}\right) ;$
	
	(ii) $Q(\cdot),H(\cdot) \in L^\infty\left(0, T ; \mathbb{S}^n\right)$, $R(\cdot) \in L^\infty\left(0, T ; \mathbb{S}^r\right)$, and $R(\cdot)>0$, $Q(\cdot) \geqslant 0$, $H(\cdot) \geqslant 0$, $Q_\Gamma(\cdot)\leqslant0 ;$
	
	(iii) $f(\cdot),\eta_1(\cdot) \in L^2\left(0, T ; \mathbb{R}^n\right) ;$
	
	(iv) $\Gamma_0 \in \mathbb{R}^{n \times n}$, $\eta_0 \in \mathbb{R}^n$, $G \in \mathbb{S}^n$ are bounded and $G \geqslant 0$, $G_{\Gamma}\leqslant 0.$
\end{Assumption} 
We next obtain the necessary and sufficient conditions for the
existence of optimal social strategy of Problem (\ref{social problem}), by using variational analysis.
\begin{Theorem}
	Assume (\ref{A1}) and (\ref{A2 of social}) hold. Then Problem (\ref{social problem}) has an optimal social strategy $u^*(\cdot)=(u^*_1(\cdot),\ldots,u^*_N(\cdot))$, $u^*_k(\cdot) \in \mathcal{U}^{ad}_k$ if and only if the following Hamiltonian system admits a set of solutions $(x^*_k(\cdot),z^*_k(\cdot),p^*_k(\cdot),k=1,\cdots,N) :$
	\begin{equation}\label{Hamiltonian system of social}
		\left\{\begin{aligned}
			d x^*_k&= \left(A x^*_k+B u^*_k+C z^*_k+f\right) d t+z^*_k d W_k, \\
			d p^*_k&= -\left[A^\top p^*_k+ Qx^*_k-Q_{\Gamma} x^{*(N)}-\bar{\eta}_1\right] d t-\left[C^\top p^*_k+H z^*_k\right] d W_k, \\
			x^*_k(T)&= \xi_k,\quad p^*_k(0)=- \left[Gx^*_k(0)-G_\Gamma x^{*(N)}(0)-\bar{\eta}_0\right],
		\end{aligned}\right.
	\end{equation}
	and the optimal strategy $u^*_k(\cdot)$ of agent $\mathcal{A}_k$ satisfies the
	stationary condition:
	\begin{equation}\label{optimality conditions of social}
		u^*_k=-R^{-1}B^\top \mathbb{E}\left[p^*_k \mid \mathcal{F}_t^k\right].
	\end{equation}
\end{Theorem}

\begin{proof}
	Suppose $u^*(\cdot)=(u^*_1(\cdot),\cdots,u^*_N(\cdot))$ is optimal social strategy of Problem (\ref{social problem}) and $(x^*_k(\cdot),\\z^*_k(\cdot),k=1,\cdots,N)$ are the
	corresponding optimal trajectories. For any $u_k(\cdot) \in \mathcal{U}^{ad}_k$ and any $\varepsilon>0$, we denote 
	$$
	u_k^\varepsilon(\cdot)=u^*_k(\cdot)+\varepsilon v_k(\cdot) \in \mathcal{U}^{ad}_k,
	$$
	where $v_k(\cdot)=u_k(\cdot)-u^*_k(\cdot)$. And we set $u^\epsilon(\cdot)\equiv(u^\epsilon_1(\cdot),\cdots,u^\epsilon_N(\cdot))$.
	
	Let $(x^\epsilon_k(\cdot),z^\epsilon_k(\cdot),k=1,\cdots,N)$ be the solution of the following perturbed state
	equation
	$$
	\left\{\begin{aligned}
		x_k^\varepsilon&= \left(A x_k^\varepsilon+B u_k^\varepsilon+C z_k^\varepsilon+f\right) d t+z_k^\varepsilon d W_k, \\
		x_k^\varepsilon(T)&= \xi_k .
	\end{aligned}\right.
	$$
	Let $\Delta x_k(\cdot)=\frac{x_k^\varepsilon(\cdot)-\bar{x}_k(\cdot)}{\varepsilon}, \Delta z_k(\cdot)=\frac{z_k^\varepsilon(\cdot)-\bar{z}_k(\cdot)}{\varepsilon}$. It can be verified that $(\Delta x_k(\cdot),\Delta z_k(\cdot),k=1,\cdots,N)$ satisfies
	$$
	\left\{\begin{aligned}
		d \Delta x_k&=\left(A \Delta x_k+B v_k+C \Delta z_k\right) d t+\Delta z_k d W_k, \\
		\Delta x_k(T)&=0 .
	\end{aligned}\right.
	$$
	Applying It\^o’s formula to $\left\langle \Delta x_k(\cdot), p^*_k(\cdot)\right\rangle$, we derive
	\begin{equation}\label{Ito maximum of social}
		\begin{aligned}
			&\sum_{k=1}^N\mathbb{E}\left[0-\left\langle \Delta x_k(0),- \left[Gx^*_k(0)-G_\Gamma x^{*(N)}(0)-\bar{\eta}_0\right]\right\rangle\right]=\sum_{k=1}^N\mathbb{E} \int_0^T\left\langle B v_k, p^*_k\right\rangle\\
			&-\left\langle \Delta x_k,Qx^*_k-Q_{\Gamma} x^{*(N)}-\bar{\eta}_1\right\rangle-\left\langle \Delta z_k, H z^*_k\right\rangle d t.
		\end{aligned}
	\end{equation}
	Then
	$$
		\mathcal{J}_{soc}\left(u^\epsilon(\cdot)\right)-\mathcal{J}_{soc}\left(u^*(\cdot)\right)=\frac{\varepsilon^2}{2} X_1+\varepsilon X_2,
	$$
	where 
	$$
	\begin{aligned}
		X_1&= \sum_{k=1}^N\mathbb{E}\left[\int_0^T \left(\Delta x_k-\Gamma_1\Delta x^{(N)}  \right)^\top Q\left(\Delta x_k-\Gamma_1\Delta x^{(N)} \right)+ v_k^\top R v_k+ \Delta z_{k}^\top H \Delta z_k d t\right. \\
		&\qquad\quad \left.+\left(\Delta x_k(0)-\Gamma_0\Delta x^{(N)}(0) \right)^\top G\left(\Delta x_k(0)-\Gamma_0\Delta x^{(N)}(0) \right)\right],
	\end{aligned}
	$$
	\begin{equation}\label{X_2 of social}
		\begin{aligned}
			X_2&= \sum_{k=1}^N\mathbb{E}  {\left[\int_0^T\left(x^*_k-\Gamma_1 x^{*(N)}-\eta_1\right)^\top Q\left(\Delta x_k-\Gamma_1\Delta x^{(N)} \right)+u_k^{*\top} R v_k+z_k^{*\top} H \Delta z_k d t\right.} \\
			&\qquad\quad \left.+\left(x_k^*(0)-\Gamma_0 x^{*(N)}(0)-\eta_0\right)^\top G\left(\Delta x_k(0)-\Gamma_0\Delta x^{(N)}(0) \right)\right]\\	
			&= \sum_{k=1}^N\mathbb{E} {\left[\int_0^T \left(Qx^*_k-Q_{\Gamma} x^{*(N)}-\bar{\eta}_1\right)^\top \Delta x_k +u_k^{*\top} R v_k+z_k^{*\top} H \Delta z_{k} d t\right.} \\
			&\qquad\quad \left.+\left(Gx^*_k(0)-G_\Gamma x^{*(N)}(0)-\bar{\eta}_0\right)^\top \Delta x_k(0)\right].
		\end{aligned}
	\end{equation}
	Due to the optimality of $u^*_k(\cdot)$, we have $\mathcal{J}_k\left(u_k^\varepsilon(\cdot) ; u^*_{-k}(\cdot)\right)-\mathcal{J}_k\left(u^*_k(\cdot) ; u^*_{-k}(\cdot)\right) \geqslant 0$. Noticing $X_1 \geqslant 0$ and the arbitrariness of $\epsilon$, we have $X_2=0$. Then, simplifying (\ref{X_2 of social}) with (\ref{Ito maximum of social}), we have
	$$
	X_2=\mathbb{E} \int_0^T\left\langle B^\top p^*_k+R u^*_k, v_k\right\rangle d t.
	$$
	Due to the arbitrariness of $v_k$, we obtain the optimal conditions (\ref{optimality conditions of social}).
\end{proof}

Similarly,  we next derive proper form for the feedback representation of optimality conditions. We still have the property $\mathbb{E}\left[x^*_j(t) \mid \mathcal{F}_t^k\right]=\mathbb{E}\left[x^*_j(t)\right]=\mathbb{E}\left[x^*_k(t)\right]$, for any $j \neq k$. And then, we have
$$
\mathbb{E}\left[x^{*(N)} \mid \mathcal{F}_t^k\right]=\frac{1}{N} x^*_k+\frac{1}{N} \sum_{j \neq k } \mathbb{E}x^*_j=\frac{1}{N} x^*_k+\frac{N-1}{N}  \mathbb{E}x^*_k.
$$
And then, we introduce the following FBSDE:
\begin{equation}\label{Hamiltonian system of social change}
	\left\{\begin{aligned}
		d x^*_k&= \left(A x^*_k+B u^*_k+C z^*_k+f\right) d t+z^*_k d W_k, \\
		d \hat{p}^*_k&= -\left\{A^\top \hat{p}^*_k+\left(Q-\frac{Q_{\Gamma}}{N}\right) x^*_k-\frac{N-1}{N}Q_{\Gamma} \mathbb{E} x^*_k-\bar{\eta}_1\right\} d t-\left[C^\top \hat{p}^*_k+H z^*_k\right] d W_k, \\
		x^*_k(T) &=\xi_k,\quad \hat{p}^*_k(0)=-\left[\left(G-\frac{G_{\Gamma}}{N}\right) x^*_k-\frac{N-1}{N}G_{\Gamma} \mathbb{E} x^*_k-\bar{\eta}_0\right],
	\end{aligned}\right.
\end{equation}
where $\hat{p}^*_k:=\mathbb{E}\left[p^*_k \mid \mathcal{F}_t^k\right]$. Similarly, we have the following result whose proof is similar to Proposition \ref{FBSDE equivalent theorem}.

\begin{Proposition}
	The FBSDE (\ref{Hamiltonian system of social}) admits a set of adapted solutions if and only if FBSDE (\ref{Hamiltonian system of social change}) admits a set of adapted solutions.
\end{Proposition}

Next, we divide the decoupling procedure into two steps.
\begin{Proposition}
	Let Assumption (\ref{A1}), (\ref{A2 of social}) hold. Let $(x^*_k(\cdot),z^*_k(\cdot),p^*_k(\cdot),k=1,\cdots,N)$ be the solution of FBSDE (\ref{Hamiltonian system of social change}). Then, we have the following relations:
	\begin{equation}
		\left\{\begin{aligned}
			x^*_k&=\Sigma \left(\hat{p}^*_k-\mathbb{E} p^*_k\right)+K \mathbb{E} p^*_k+\varphi^*_k,\\
			z^*_k&=\left(I_n+\Sigma H\right)^{-1} \left(\beta^*_k-\Sigma C^\top \hat{p}^*_k\right),
		\end{aligned}\right.
	\end{equation}
	and $\Sigma(\cdot),K(\cdot),\varphi^*_k(\cdot)$ are solutions of the following equations, respectively, 
	\begin{equation}\label{Sigma equation of social}
		\left\{\begin{aligned}
			&\dot{\Sigma}-A \Sigma-\Sigma A^\top-\Sigma\left(Q-\frac{Q_{\Gamma}}{N}\right) \Sigma+B R^{-1} B^\top+C \left(I_n+\Sigma H\right)^{-1} \Sigma C^\top=0, \\
			&\Sigma(T)=0,
		\end{aligned}\right.
	\end{equation}
	\begin{equation}\label{K equation of social}
		\left\{\begin{aligned}
			& \dot{K}-A K-K A^\top-K\left(Q-Q_\Gamma\right) K+B R^{-1} B^\top+C \left(I_n+\Sigma H\right)^{-1} \Sigma C^\top=0, \\
			& K(T)=0,
		\end{aligned}\right.
	\end{equation}
	\begin{equation}\label{varphi equation of social}
		\left\{\begin{aligned}
			d\varphi^*_k&= \left\{\left[A+\Sigma\left(Q-\frac{Q_\Gamma}{N}\right)\right] \varphi^*_k -\left[\left(\Sigma-K\right)Q-\left(K-\frac{\Sigma}{N}\right)Q_\Gamma\right]\mathbb{E}\varphi^*_k \right.\\
			&\qquad +C\left(I_n+\Sigma H\right)^{-1}\beta^*_k-K \bar{\eta}_1+f\bigg\}dt+\beta^*_k d W_k,\\
			\varphi^*_k(T)&= \xi_k.
		\end{aligned}\right.
	\end{equation}
\end{Proposition}

\begin{Proposition}\label{decouple proposition 4}
	Let Assumption (\ref{A1}), (\ref{A2 of social}) hold. Let $(x^*_k(\cdot),z^*_k(\cdot),p^*_k(\cdot),k=1,\cdots,N)$ be the solution of FBSDE (\ref{Hamiltonian system of social change}). Then, we have the following relations:
	\begin{equation}\label{decouple form of social 4}
		\hat{p}^*_k=\Pi \left(x^*_k-\mathbb{E}x^*_k\right)+M \mathbb{E}x^*_k+\zeta^*_k,
	\end{equation}
	where $\Pi(\cdot),M(\cdot),\zeta^*_k(\cdot)$ are solutions of the following equations, respectively, 
	\begin{equation}\label{Pi equation of social}
		\left\{\begin{aligned}
			&\dot{\Pi}+\Pi A^\top+A \Pi-\Pi\left[B R^{-1} B^\top +C\left(I_n+\Sigma H\right)^{-1}\Sigma C^\top\right]\Pi+\left(Q-\frac{Q_{\Gamma}}{N}\right)=0, \\
			&\Pi(0)=-\left(G-\frac{G_{\Gamma}}{N}\right),
		\end{aligned}\right.
	\end{equation}
	\begin{equation}\label{M equation of social}
		\left\{\begin{aligned}
			&\dot{M}+M A^\top+A M-M\left[BR^{-1}B^\top+C\left(I_n+\Sigma H\right)^{-1}\Sigma C^\top\right]M+ \left(Q-Q_{\Gamma}\right)=0,\\
			&M(0)=-\left(G-G_{\Gamma}\right),
		\end{aligned}\right.
	\end{equation}
	\begin{equation}\label{zeta equation of social}
		\left\{\begin{aligned}
			d \zeta^*_k&= \bigg\{\left\{\Pi\left[BR^{-1}B^\top +C\left(I_n+\Sigma H\right)^{-1}\Sigma C^\top \right]-A^\top\right\} \zeta^*_k-\Pi C\left(I_n+\Sigma H\right)^{-1}\beta^*_k \\
			&\qquad +\left(M-\Pi\right)\left[BR^{-1}B^\top +C\left(I_n+\Sigma H\right)^{-1}\Sigma C^\top \right]\mathbb{E}\zeta^*_k-Mf-\bar{\eta}_1\\
			&\qquad +\left(\Pi-M\right)C\left(I_n+\Sigma H\right)^{-1}\mathbb{E}\beta^*_k\bigg\}dt\\
			&\quad +\bigg\{-\left(\Pi+H\right)\left(I_n+\Sigma H\right)^{-1}\beta^*_k-\left(I_n-\Pi\Sigma \right)\left(I_n+H\Sigma \right)^{-1} C^\top\\
			&\qquad\ \times\left\{\left(I_n-\Pi\Sigma\right)^{-1}\left[\Sigma\left(\varphi^*_k-\mathbb{E}\varphi^*_k\right)+\left(\zeta^*_k-\mathbb{E}\zeta^*_k\right)\right]\right.\\
            &\qquad\ \left.+\left(I_n-MK\right)^{-1}\left(M\mathbb{E}\varphi^*_k+\mathbb{E}\zeta^*_k\right)\right\}\bigg\}dW_k,\\
			\zeta^*_k(0)&= \bar{\eta}_0.
		\end{aligned}\right.
	\end{equation}
\end{Proposition}
The above discussion can be summarized as the following theorem.
\begin{Theorem}
	Let Assumption (\ref{A1}), (\ref{A2 of social}) hold. Riccati equations (\ref{Sigma equation of social}), (\ref{K equation of social}), (\ref{Pi equation of social}), (\ref{M equation of social}), BSDE (\ref{varphi equation of social}), SDE (\ref{zeta equation of social}) admit unique solutions, respectively. In addition, the optimal strategy of agent $\mathcal{A}_k$ has a feedback form as follows
	\begin{equation}\label{optimal social strategy}
		u^*_k=-R^{-1} B^\top\big[\Pi x^*_k+\left(M-\Pi\right) \mathbb{E}x^*_k+\zeta^*_k\big],
	\end{equation}
	where
	\begin{equation}
		\left\{\begin{aligned}
			d x^*_k&= \left[\left(A-BR^{-1} B^\top \Pi\right) x^*_k-BR^{-1} B^\top \left(M-\Pi\right) \mathbb{E}x^*_k\right.\\
            &\qquad \left.-BR^{-1} B^\top \zeta^*_k+C z^*_k+f\right] d t+z^*_k d W_k, \\
			x^*_k(T)&= \xi_k,
		\end{aligned}\right.
	\end{equation}
	\begin{equation}
		\left\{\begin{aligned}
			d \mathbb{E}x^*_k&= \left\{\left(A-BR^{-1} B^\top M-C \left(I_n+\Sigma H\right)^{-1}\Sigma C^\top M\right) \mathbb{E}x^*_k \right.\\
			&\qquad \left.-\left[BR^{-1} B^\top +C \left(I_n+\Sigma H\right)^{-1}\Sigma C^\top \right]\mathbb{E}\zeta^*_k+C \left(I_n+\Sigma H\right)^{-1} \mathbb{E}\beta^*_k+f\right\} d t, \\
			\mathbb{E}x^*_k(T)&= \mathbb{E}\xi_k.
		\end{aligned}\right.
	\end{equation}
\end{Theorem}

\subsection{Comparison with classical mean field social optima problem}

We now review the classical results of mean field team problem for comparison to this work. Consider the large-population of Problem (\ref{social problem}) case with $C(\cdot)=0$. Then, let $N$ intends to infinity, we introduce the following equation 
\begin{equation}\label{Sigma equation of social limit}
	\left\{\begin{aligned}
		&\dot{\bar{\Sigma}}-A \bar{\Sigma}-\bar{\Sigma} A^\top -\bar{\Sigma}Q \bar{\Sigma}+B R^{-1} B^\top=0, \\
		&\bar{\Sigma}(T)=0,
	\end{aligned}\right.
\end{equation}

\begin{equation}\label{K equation of social limit}
	\left\{\begin{aligned}
		& \dot{\bar{K}}-A \bar{K}-\bar{K} A^\top -\bar{K}\left(Q-Q_{\Gamma}\right) \bar{K}+B R^{-1} B^\top=0, \\
		& \bar{K}(T)=0,
	\end{aligned}\right.
\end{equation}

\begin{equation}\label{varphi equation of social limit}
	\left\{\begin{aligned}
		d\bar{\varphi}^*_k&= \left\{\left(A+\bar{\Sigma}Q\right) \bar{\varphi}^*_k -\left[\left(\bar{\Sigma}-\bar{K}\right)Q-\bar{K}Q_{\Gamma}\right]\mathbb{E}\bar{\varphi}^*_k -\bar{K} \bar{\eta}_1+f\right\}dt+ \bar{\beta}^*_k d W_k,\\
		\bar{\varphi}^*_k(T)&= \xi_k,
	\end{aligned}\right.
\end{equation}
\begin{equation}\label{Pi equation of social limit}
	\left\{\begin{aligned}
		&\dot{\bar{\Pi}}+\bar{\Pi} A^\top +A \bar{\Pi}-\bar{\Pi}B R^{-1} B^\top \bar{\Pi}+Q=0, \\
		&\bar{\Pi}(0)=-G,
	\end{aligned}\right.
\end{equation}
\begin{equation}\label{M equation of social limit}
	\left\{\begin{aligned}
		&\dot{\bar{M}}+\bar{M} A^\top +A \bar{M}-\bar{M}BR^{-1}B^\top \bar{M}+ \left(Q-Q_{\Gamma}\right)=0,\\
		&\bar{M}(0)=-\left(G-G_{\Gamma}\right),
	\end{aligned}\right.
\end{equation}
\begin{equation}\label{zeta equation of social limit}
	\left\{\begin{aligned}
		d \bar{\zeta}^*_k&= \left[\left(\bar{\Pi}BR^{-1}B^\top -A^\top \right) \bar{\zeta}^*_k+\left(\bar{M}-\bar{\Pi}\right)BR^{-1}B^\top \mathbb{E}\bar{\zeta}^*_k-\bar{M}f-\bar{\eta}_1\right]dt\\
		&\qquad -\left[\left(\bar{\Pi}+H\right)\left(I_n+\bar{\Sigma}H\right)^{-1}\bar{\beta}^*_k\right]dW_k,\\
		\bar{\zeta}^*_k(0)&= \bar{\eta}_0.
	\end{aligned}\right.
\end{equation}

And we have the following result.
\begin{Proposition}\label{Proposition estimation of social Sigma limit}
	Let Assumption (\ref{A1}), (\ref{A2 of social}) hold and $C(\cdot)=0$, then we have $K(\cdot)=\bar{K}(\cdot)$, $M(\cdot)=\bar{M}(\cdot)$ and
	\begin{equation}\label{estimation of social Riccati limit}
	\sup_{0\leqslant t \leqslant T}\left\|\Sigma(t)-\bar{\Sigma}(t)\right\|=O\left(\frac{1}{N}\right),\quad \sup_{0\leqslant t \leqslant T}\left\|\Pi(t)-\bar{\Pi}(t)\right\|=O\left(\frac{1}{N}\right).
    \end{equation} 
\end{Proposition}
\begin{proof}
	Combining (\ref{K equation of social}) with (\ref{K equation of social limit}), $K(\cdot)=\bar{K}(\cdot)$, $M(\cdot)=\bar{M}(\cdot)$ is obvious. Using the Theorem 4 in \cite{Huang-Zhou-2020}, let $\delta_N=\int_{0}^{T}\left\|\Sigma \left(Q-\frac{Q_\Gamma}{N}\right)\Sigma \right\|dt$, we have the first estimation of (\ref{estimation of social Riccati limit}). Similarly,  $z_N=\frac{G_\Gamma}{N}-G,z=-G$ and $\delta_N=\int_{0}^{T}\left\|Q-\frac{Q_\Gamma}{N}\right\|dt$, we have the second estimation of (\ref{estimation of social Riccati limit}). 
\end{proof}
\begin{Proposition}\label{Proposition estimation of social varphi limit}
	Let Assumption (\ref{A1}), (\ref{A2 of social}) hold and $C(\cdot)=0$, then we have 
	\begin{equation}\label{estimation of social varphi limit}
		\begin{aligned}
		&\mathbb{E}\int_0^T\left\|\varphi^*_k-\bar{\varphi}^*_k\right\|^2dt=O\left(\frac{1}{N^2}\right),\\
		&\mathbb{E}\int_0^T\left\|\beta^*_k-\bar{\beta}^*_k\right\|^2dt=O\left(\frac{1}{N^2}\right),\\
		&\mathbb{E}\int_0^T\left\|\zeta^*_k-\bar{\zeta}^*_k\right\|^2dt=O\left(\frac{1}{N}\right).
	\end{aligned}
	\end{equation} 
\end{Proposition}
\begin{proof}
	By (\ref{varphi equation of social}) and (\ref{varphi equation of social limit}), we have 
	$$
	\left\{\begin{aligned}
		d\left(\varphi^*_k-\bar{\varphi}^*_k\right)&= \left[\left(\Sigma-\bar{\Sigma}\right)Q\varphi^*_k+\left(A+\Sigma Q\right)\left(\varphi^*_k-\bar{\varphi}^*_k\right)-\Sigma\frac{Q_\Gamma}{N}\varphi^*_k+\frac{\Sigma}{N}Q_\Gamma\mathbb{E}\varphi^*_k\right]dt\\
        &\qquad +\left(\beta^*_k-\bar{\beta}^*_k\right)dW_k, \\
		\left(\varphi^*_k-\bar{\varphi}^*_k\right)(T)&= 0.
	\end{aligned}\right.
	$$
	Applying the standard estimate of BSDE (\cite{Li-Sun-Xiong-2019}), we derive
	$$
	\begin{aligned}
	\mathbb{E}&\int_0^T\left\|\varphi^*_k-\bar{\varphi}^*_k\right\|^2dt+\mathbb{E}\int_0^T\left\|\beta^*_k-\bar{\beta}^*_k\right\|^2dt\\
    &\leqslant C\left\{ \mathbb{E}\int_0^T\left\|\left(\Sigma\bar{\Sigma}\right)Q\varphi^*_k\right\|^2dt
    +\mathbb{E}\int_0^T\left\|\Sigma\frac{Q_\Gamma}{N}\varphi^*_k\right\|^2dt +\mathbb{E}\int_0^T\left\|\frac{\Sigma}{N}Q_\Gamma\mathbb{E}\varphi^*_k\right\|^2dt\right\}\\
	&\leqslant C\left\{ \left[\left(\sup_{0\leqslant t \leqslant T}\left\|\Sigma-\bar{\Sigma}\right\|\right)^2+\left(\sup_{0\leqslant t \leqslant T}\left\|\frac{\Sigma}{N}\right\|\right)^2 \right]\left(\mathbb{E}\int_0^T\left\|\varphi^*_k\right\|^2dt\right)\right.\\
	&\qquad\quad \left.+\left(\sup_{0\leqslant t \leqslant T}\left\|\frac{\Sigma}{N}\right\|\right)^2\mathbb{E}\int_0^T\left\|\mathbb{E}\varphi^*_k\right\|^2dt\right\}=O\left(\frac{1}{N^2}\right).
    \end{aligned}
	$$
	Therefore, we get the first and second estimation of (\ref{estimation of social varphi limit}).
	By (\ref{zeta equation of social}) and (\ref{zeta equation of social limit}), we have 
	$$
	\left\{\begin{aligned}
		d\left(\zeta^*_k-\bar{\zeta}^*_k\right)&= \left[\left(\Pi-\bar{\Pi}\right)BR^{-1}B^\top \zeta^*_k+\left(\bar{\Pi}BR^{-1}B^\top -A^\top\right)\left(\zeta^*_k-\bar{\zeta}^*_k\right)\right.\\
		&\qquad \left. +\left(-\Pi+\bar{\Pi}\right)BR^{-1}B^\top \mathbb{E} \zeta^*_k+\left(\bar{M}-\bar{\Pi}\right)BR^{-1}B^\top \left(\mathbb{E} \zeta^*_k-\mathbb{E} \bar{\zeta}^*_k\right) \right]dt\\
		&\quad -\left\{\left[\left(\Pi-H\right)\left(I_n+\Pi H\right)^{-1}-\left(\bar{\Pi}-H\right)\left(I_n+\bar{\Pi} H\right)^{-1}\right]\beta^*_k \right.\\
		&\qquad\ \left. +\left(\bar{\Pi}-H\right)\left(I_n+\bar{\Pi} H\right)^{-1}\left(\beta^*_k-\bar{\beta}^*_k\right)\right\}dW_k, \\
		\left(\zeta^*_k-\bar{\zeta}^*_k\right)(0)&= 0.
	\end{aligned}\right.
	$$
	Due to 
	$$
	\begin{aligned}
	&\left(I_n+\bar{\Pi} H\right)\left[\left(I_n+\Pi H\right)^{-1}-\left(I_n+\bar{\Pi} H\right)^{-1}\right]\\
	&=\left(I_n+\bar{\Pi} H\right)\left(I_n+\Pi H\right)^{-1}-I_n
	=\left(\bar{\Pi}-\Pi\right)H\left(I_n+\Pi H\right)^{-1},
	\end{aligned}
	$$
	we have
	\begin{equation}\label{fact}
		\left(I_n+\Pi H\right)^{-1}-\left(I_n+\bar{\Pi} H\right)^{-1}=\left(I_n+\bar{\Pi} H\right)^{-1}\left(\bar{\Pi}-\Pi\right)H\left(I_n+\Pi H\right)^{-1}
	\end{equation}
	Applying the standard estimate of SDE and the fact (\ref{fact}), we derive
	$$
	\begin{aligned}
		\mathbb{E}\int_0^T\left\|\zeta^*_k-\bar{\zeta}^*_k\right\|^2dt &\leqslant C \left\{ \mathbb{E}\int_0^T\left\|\left(\Pi-\bar{\Pi}\right)BR^{-1}B^\top \zeta^*_k+\left(\bar{\Pi}-\Pi\right)BR^{-1}B^\top \mathbb{E}\zeta^*_k\right\|^2dt \right.\\
		&\qquad\ +\mathbb{E}\int_0^T\left\|\left\{\left(\Pi-H\right)\left[\left(I_n+\Pi H\right)^{-1}-\left(I_n+\bar{\Pi} H\right)^{-1}\right]\right.\right.\\
        &\qquad\ \left.\left.\left.+\left(\Pi-\bar{\Pi}\right)\left(I_n+\bar{\Pi} H\right)^{-1}\right\}\beta^*_k \left(\bar{\Pi}-H\right)\left(I_n+\bar{\Pi} H\right)^{-1}\left(\beta^*_k-\bar{\beta}^*_k\right)\right\|^2dt\right\}\\
		&\leqslant C\left\{\left(\sup_{0\leqslant t \leqslant T}\left\|\Pi-\bar{\Pi}\right\|\right)^2\left(\mathbb{E}\int_0^T\left\|\beta^*_k\right\|^2dt\right) +\mathbb{E}\int_0^T\left\|\beta^*_k-\bar{\beta}^*_k\right\|^2dt\right\}\\
        &= O\left(\frac{1}{N^2}\right).
	\end{aligned}
	$$
	Therefore, we get the third estimation of (\ref{estimation of social varphi limit}). Then, we complete the proof.
\end{proof}
Then, let $N$ intends to infinity, we introduce the following equation 
\begin{equation}
	\left\{\begin{aligned}
		d \bar{x}^*_k&= \left[\left(A-BR^{-1} B^\top \bar{\Pi}\right) \bar{x}^*_k-BR^{-1} B^\top \left(\bar{M}-\bar{\Pi}\right) \mathbb{E}\bar{x}^*_k-BR^{-1} B^\top \bar{\zeta}^*_k+f\right] d t\\
		&\quad +\bar{z}^*_k d W_k, \\
		\bar{x}^*_k(T)&= \xi_k,
	\end{aligned}\right.
\end{equation}
and have the following result whose proof is similar to that of Proposition \ref{Proposition estimation of social varphi limit}.
\begin{Proposition}\label{Proposition estimation of social x limit}
	Let Assumption (\ref{A1}), (\ref{A2 of social}) hold and $C(\cdot)=0$, then we have 
	\begin{equation}\label{estimation of social x limit}
		\begin{aligned}
			&\mathbb{E}\int_{0}^{T}\left\|\bar{x}^*_k-x^*_k\right\|^2dt=O\left(\frac{1}{N^2}\right),\\
			&\mathbb{E}\int_{0}^{T}\left\|\bar{z}^*_k-z^*_k\right\|^2dt=O\left(\frac{1}{N^2}\right),\\
		\end{aligned}
	\end{equation} 
\end{Proposition}

Then, we obtain the decentralized control law for agent $\mathcal{A}_k$ of our large-population system:
\begin{equation}\label{optimal social limit strategy}
	\bar{u}^*_k=-R^{-1} B^\top \left[\bar{\Pi} \bar{x}^*_k+\left(\bar{M}-\bar{\Pi}\right) \mathbb{E}\bar{x}^*_k+\bar{\zeta}^*_k\right].
\end{equation}

By the direct approach \cite{Li-Wu-2023}, the following decentralized strategies of Problem \ref{social problem} with $C(\cdot)=0$ are obtained:
\begin{equation}\label{optimal social limit strategy LiMZ}
	\widetilde{u}^*_k=-R^{-1} B^\top \left[\bar{\Pi} \bar{x}^*_k+\widetilde{M} \mathbb{E}\bar{x}^*_k+\bar{\zeta}^*_k\right],
\end{equation}
where $\widetilde{M}(\cdot)$ satisfies
\begin{equation}\label{M widetilde equation of social limit}
	\left\{\begin{aligned}
		&\dot{\widetilde{M}}+\widetilde{M} A^\top +A \widetilde{M}-\widetilde{M}BR^{-1}B^\top \bar{\Pi}-\bar{\Pi}BR^{-1}B^\top \widetilde{M}-\widetilde{M}BR^{-1}B^\top \widetilde{M}-Q_{\Gamma}=0,\\
		&\bar{M}(0)=G_{\Gamma}.
	\end{aligned}\right.
\end{equation}
Such set of decentralized strategies (\ref{optimal social limit strategy LiMZ}) is further shown to be an $\epsilon$-Nash equilibrium with respect to $\mathcal{U}_k^{c}$, i.e., 
$$
\mathcal{J}_k\left(\widetilde{u}^*_1(\cdot),\cdots,\widetilde{u}^*_N(\cdot)\right) \leqslant \inf _{u_k(\cdot) \in \mathcal{U}_{k}^c}\mathcal{J}_k\left(u_1(\cdot),\cdots,u_N(\cdot)\right)+\varepsilon,
$$
where $\mathcal{U}^{c}_k= \bigg\{u_k(\cdot) \mid u_k(t) \text { is adapted to } \mathcal{F}_t \text { and } \mathbb{E} \int_0^T\left|u_k(t)\right|^2 d t<\infty\bigg\}$ and $\epsilon=O(\frac{1}{\sqrt{N}})$.

We can easily verify $\widetilde{M}(\cdot)+\bar{\Pi}(\cdot)=\bar{M}(\cdot)$. Therefore, we have $\widetilde{u}^*_k(\cdot)=\bar{u}^*_k(\cdot)$. That is to say, the control gains of (\ref{optimal social limit strategy}) and (\ref{optimal social limit strategy LiMZ}) coincide for the infinite population case. 
\begin{Remark}
	The set of decentralized strategies (\ref{optimal social limit strategy}) is an exact Nash equilibrium with respect to $\mathcal{U}_K^{ad}, K=1,\cdots,N$. It is applicable for arbitrary number of players. In contrast, the set of decentralized strategies (\ref{optimal social limit strategy LiMZ}) is an asymptotic Nash equilibrium with respect to $\mathcal{U}_k^c, k=1,\cdots,N$. It is only applicable for the
	large-population case.
\end{Remark}

\section{Numerical examples}

In this section, we give a numerical example with certain particular coefficients to simulate our theoretical results. We set the number of agents to 30, i.e., $N=30$ and the terminal time is $1$. The simulation parameters are given as follows: $A=0.1,B=2,C=1,Q=1,R=5,H=1,G=2,\Gamma_1=0.5,\eta_1=1,\Gamma_0=0.5,\eta_0=1$. And for $i=1,\cdots,N$, $\xi_i=W_i(T)$. By the  Euler's method, we plot the solution curves of Riccati equations (\ref{Sigma equation}), (\ref{K equation}), (\ref{Pi equation}) and (\ref{M equation})
in Figure \ref{fig:Riccati}. By the Monte Carlo method, the figures of $\zeta^*_i(\cdot)$ and optimal state $x^*_i(\cdot)$ are shown in Figures \ref{fig:zeta^*_i} and \ref{fig:x^*_i}. Further, we also generate the dynamic simulation of optimal control $u^*_i(\cdot)$, shown in Figure \ref{fig:u^*_i}.

\begin{figure}[htbp] 
	\centering  
	\begin{minipage}{0.55\textwidth} 
		\centering  
		\includegraphics[width=\textwidth]{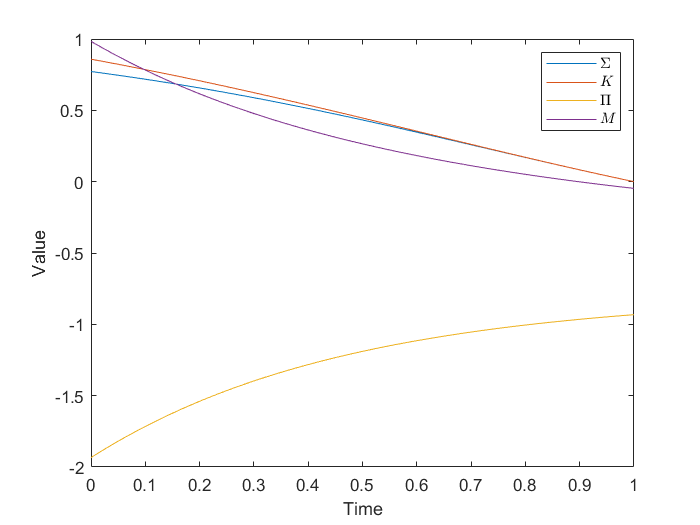} 
		\caption{\centering The solution curve of $\Sigma(\cdot)$, $K(\cdot)$, $\Pi(\cdot)$ and $M(\cdot)$} 
		\label{fig:Riccati} 
	\end{minipage}  
	\hspace{0.05\textwidth} 
	\begin{minipage}{0.55\textwidth}  
		\centering  
		\includegraphics[width=\textwidth]{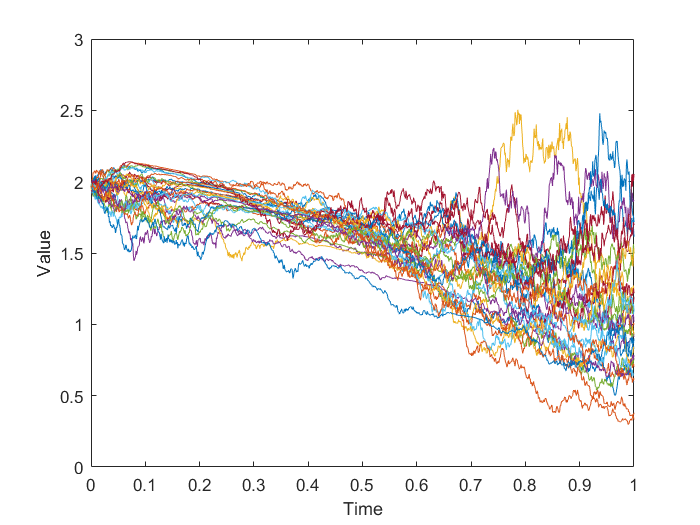} 
		
		\caption{\centering The solution curve of $\zeta^*_i(\cdot)$, $i=1,\cdots,30$} 
		\label{fig:zeta^*_i} 
	\end{minipage}  
\end{figure}  

\begin{figure}[htbp] 
	\centering  
	\begin{minipage}{0.55\textwidth} 
		\centering  
		\includegraphics[width=\textwidth]{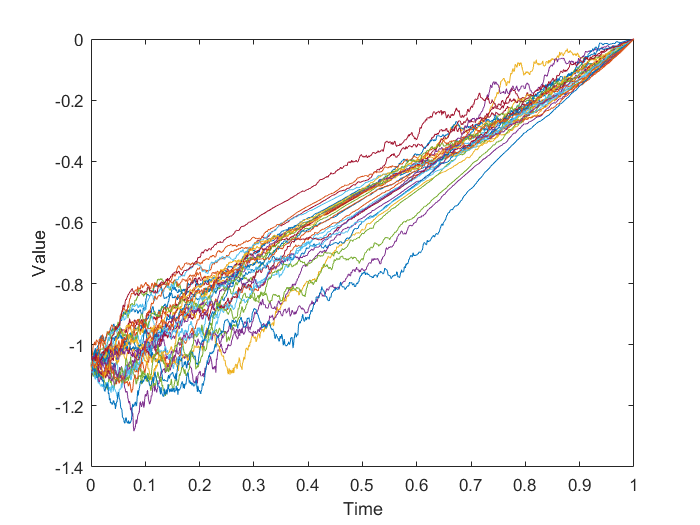} 
		\caption{\centering The solution curve of $x^*_i(\cdot)$, $i=1,\cdots,30$}  
		\label{fig:x^*_i} 
	\end{minipage}  
	\hspace{0.05\textwidth} 
	\begin{minipage}{0.55\textwidth}  
		\centering  
		\includegraphics[width=\textwidth]{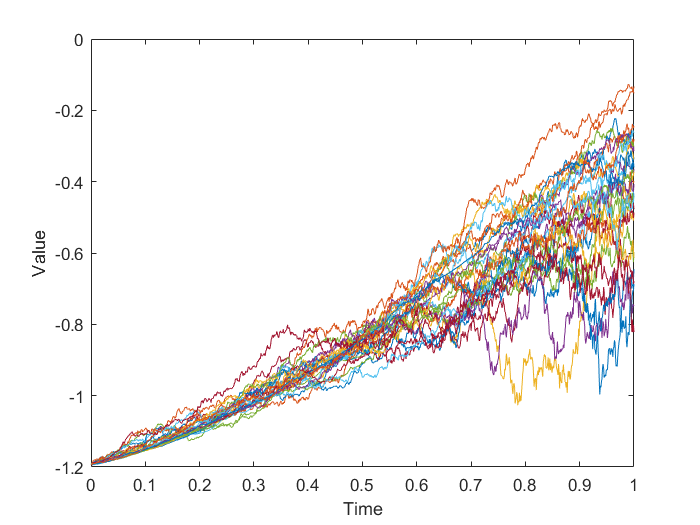} 
		
		\caption{\centering The solution curve of $u^*_i(\cdot)$, $i=1,\cdots,30$}  
		\label{fig:u^*_i}   
	\end{minipage}  
\end{figure}  

\section{Conclusion}

In this paper, we have studied an LQ mean field games and teams problems with linear backward stochastic differential equation. We adopts the backward separation approach (\cite{Wang-Wu-Xiong-2018}) to solve the problems and obtain the decentralized optimal strategy. Our present work suggests various future research directions. For example, (i) To study the backward MFG with indefinite control weight and this will formulate the mean-variance analysis with relative performance in our setting;
(ii) To study the backward MFG with integral-quadratic constraint, we can attempt to adopt the method of Lagrange multipliers and the Ekeland variational principle; (iii) To consider the direct method to solve mean field problem with the state equation contains state average term.  
We plan to study these issues in our future works.

\end{document}